\documentclass[10pt]{amsart}
\usepackage{amssymb}
\usepackage{graphicx}
\usepackage{xcolor} 
\usepackage{tensor}
\usepackage{fullpage} 
\usepackage{amsmath}
\usepackage{amsthm}
\usepackage{verbatim}
\usepackage{hyperref}
\usepackage{enumitem}
\setlist[enumerate]{leftmargin=1.5em}
\setlist[itemize]{leftmargin=1.5em}

\providecommand{\bysame}{\leavevmode\hbox to3em{\hrulefill}\thinspace}
\providecommand{\MR}{\relax\ifhmode\unskip\space\fi MR }
\providecommand{\MRhref}[2]{%
	\href{http://www.ams.org/mathscinet-getitem?mr=#1}{#2}
}
\providecommand{\href}[2]{#2}

\usepackage{xcolor}
\definecolor{purple}{rgb}{0.5, 0, 1}
\definecolor{orange}{rgb}{1,.5,0}

\providecommand{\bysame}{\leavevmode\hbox to3em{\hrulefill}\thinspace}
\providecommand{\MR}{\relax\ifhmode\unskip\space\fi MR }
\providecommand{\MRhref}[2]{%
	\href{http://www.ams.org/mathscinet-getitem?mr=#1}{#2}
}
\providecommand{\href}[2]{#2}

\usepackage{seqsplit} 
\definecolor{green}{rgb}{0,0.8,0} 

\newtheorem{maintheorem}{Theorem}

\newtheorem{theorem}{Theorem}[section]

\newtheorem{lemma}[theorem]{Lemma}

\theoremstyle{definition}

\theoremstyle{remark}
\newtheorem{remark}[theorem]{Remark}

\numberwithin{equation}{section}
\newcommand{\nrm}[1]{\Vert#1\Vert}

\newcommand{\brk}[1]{\langle#1\rangle}

\newcommand{\nnrm}[1]{{\vert\kern-0.25ex\vert\kern-0.25ex\vert #1 
		\vert\kern-0.25ex\vert\kern-0.25ex\vert}}

\newcommand{\supp}{{\mathrm{supp}}\,}

\newcommand{\lap}{\Delta}

\newcommand{\ud}{\mathrm{d}}
\newcommand{\rd}{\partial}
\newcommand{\nb}{\nabla}

\newcommand{\ohp}{\overline{\bbR^2_+}}

\newcommand{\alp}{\alpha}
\newcommand{\bt}{\beta}
\newcommand{\gmm}{\gamma}

\newcommand{\dlt}{\delta}

\newcommand{\tht}{\theta}

\newcommand{\bbR}{\mathbb R}

\newcommand{\varep}{\varepsilon}

\begin{document}
	
	\title{Wellposedness of inviscid SQG in the half-plane}

	\date{\today}
	
	\renewcommand{\thefootnote}{\fnsymbol{footnote}}

	\renewcommand{\thefootnote}{\arabic{footnote}}


	\author{In-Jee Jeong}
	\address{Department of Mathematical Sciences and RIM, Seoul National University, 1 Gwanak-ro, Gwanak-gu, Seoul 08826, South Korea.}
	\email{injee$\_$j@snu.ac.kr}

	\author{Junha Kim}
	\address{Department of Mathematics, Ajou University.}
	\email{junha02@ajou.ac.kr}

	\author{Hideyuki Miura}
	\address{Department of Mathematics, Tokyo Institute of Technology.}
	\email{miura.h.aj@m.titech.ac.jp}

	\maketitle
	
	
	\begin{abstract}
		We consider the SQG equation without dissipation on the half-plane with Dirichlet boundary condition, and prove local wellposedness in the spaces $W^{3,p}$ and $C^{2,\beta}$ for any $1<p<\infty$ and $0<\beta<1$. We complement this wellposedness by showing that for generic $C^{\infty}_{0}$ initial data, the unique corresponding  solution does not belong to $W^{3,\infty}$.  
	\end{abstract}

	
	\section{Introduction}

	\subsection{Inviscid SQG in the half-plane}
	
	We consider the Cauchy problem for the inviscid surface quasi-geostrophic (SQG) equation in the upper half-plane $\bbR^2_+ = \{ (x_1,x_2) \in \bbR^2 ; x_{2} > 0 \}$:
	\begin{equation}  \label{eq:SQG} \tag{SQG} \rd_t\tht + u\cdot \nb \tht = 0,   
	\end{equation} where the velocity is given by  \begin{equation}\label{eq:Biot-Savart-SQG} \tag{BS}
		\begin{split}
			u(t,x) = P.V. \int_{ \bbR^2_+ } \left[ \frac{(x-y)^\perp}{|x-y|^{3}} - \frac{(x-\bar{y})^\perp}{|x-\bar{y}|^{3}}   \right] \tht(t,y) \, \ud y,
		\end{split}
	\end{equation} with $\bar{y}:=(y_1,-y_2)$ for $y=(y_1,y_2)$. That is, up to a multiplicative constant that we neglect, $u = R^{\perp}\bar{\tht}$, where $R^{\perp} = (-R_{2}, R_{1})$ with $R_{i} = \rd_{i}(-\lap)^{-1/2}$ and $\bar{\tht}$ is the odd extension of $\tht$ to $\bbR^{2}$. As we shall explain below, to study strong solutions for \eqref{eq:SQG} in the half-plane, it is essential to impose the Dirichlet boundary condition \begin{equation}\label{eq:DBC} \tag{DBC} 
		\begin{split}
			\tht(t, x_1,0) \equiv 0, 
		\end{split}
	\end{equation} which will be assumed throughout the paper. Since the velocity satisfies the slip boundary conditions, the condition \eqref{eq:DBC} propagates in time. \\
	
	While dissipative versions of \eqref{eq:SQG} have been studied quite extensively in domains with boundary (\cite{CI1,CI2,CNgu,CNgu2,StoVa,Ign,Iwa2,Iwa1}), culminating in the  global wellposedness result with critical dissipation (\cite{CIN}), for the inviscid SQG, the only wellposedness for strong solutions that we are aware of is the one in Constantin--Nguyen \cite{CNgu} which gives local wellposedness in $H^{1}_{0} \cap W^{2,p}$ for any $2<p<\infty$. While the proof in \cite{CNgu} is given for any planar bounded domain with smooth boundary, it can be easily adapted to the  half-plane case. 
	 Given this $H^{1}_{0} \cap W^{2,p}$ wellposedness result, it is natural to ask whether higher Sobolev or H\"older regularity of the initial data can propagate in time for this $H^{1}_{0} \cap W^{2,p}$ solution. More precisely, we investigate what are the smallest Sobolev and H\"older spaces in which \eqref{eq:SQG} is locally wellposed with \eqref{eq:DBC}. Furthermore, we study  the \textit{maximal regularity} of the unique local-in-time strong solution, for a generic initial datum which is compactly supported and $C^{\infty}$ up to the boundary.

	\subsection{Main Results}\label{sec_results}
	
	Our first result gives local wellposedness of data belonging to either $C^{2,\bt}$ ($0<\bt<1$) or $W^{3,p}$ ($1<p<\infty$), under the Dirichlet boundary condition. Both of these spaces embed into $W^{2,q}$ with some $q>2$, and therefore it is a natural extension of the inviscid wellposedness result of \cite{CNgu}. 
	
	Before stating the result, let us introduce a simplifying notation: if $X$ is a Banach space of functions on $\ohp$ which embeds in $C(\ohp)$, we write $X_{0}$ to be the subspace of $X$ which are compactly supported and vanishes on the boundary of the upper half-plane. (We are using the notation $\ohp$ to emphasize that the regularity holds uniformly up to the boundary.)
	 
	\begin{maintheorem}\label{thm:loc} We have local wellposedness for \eqref{eq:SQG} on $\bbR^{2}_{+}$ with \eqref{eq:DBC} in $C^{2,\bt}_{0}(\ohp)$ for any $0<\bt<1$ and also in $W^{3,p}_{0}(\ohp)$ for any $1 < p < \infty$. Namely, for any initial data $\tht_{0} \in C^{2,\bt}_{0}(\ohp)$ ($W^{3,p}_{0}(\ohp)$, resp.), there exist $\dlt>0$ and a unique corresponding solution satisfying $\tht \in L^{\infty}([0,\dlt]; C^{2,\bt}_{0}(\ohp) )$ (${C}([0,\dlt]; W^{3,p}_{0}(\ohp) )$, resp.). 
	\end{maintheorem}

\begin{remark}
	A few remarks are in order. \begin{itemize}
		\item As in the $\bbR^{2}$ case,   \begin{equation*}
			\begin{split}
				\int_0^T \nrm{\nb \tht(t,\cdot)}_{L^\infty(\bbR^2_+)} dt < \infty 
			\end{split}
		\end{equation*} controls singularity formation at $t=T$. Furthermore, the uniqueness statement requires just $\tht \in L^{\infty}([0,\dlt];C^{1}(\ohp))$ (the proofs in \cite{Bed,KJ-SQG} for $\bbR^2$ can be easily adapted to the half-plane case). 
		\item  As a direct consequence of this wellposedness result, $C^\infty_0$ initial datum has the corresponding unique solution belonging to to $C^{2,\bt} \cap W^{3,p}$ for any $\bt<1$ and $p<\infty$ during its lifespan.
	\end{itemize}
\end{remark} 
	
The next result complements the first main result by showing that there is \textit{nonexistence} of solutions in $W^{3,\infty}_{0}(\ohp)$ (and hence also in $C^{3}_{0}(\ohp)$), even if the initial data is infinitely differentiable up to the boundary. 
	\begin{maintheorem}\label{thm:nonexist}
		There exist initial data in $\tht_0 \in C^{\infty}_0(\overline{\bbR^2_+})$ such that the corresponding unique solution to \eqref{eq:SQG} given in Theorem \ref{thm:loc} does not belong to $L^\infty([0,\delta];W^{3,\infty}(\overline{\bbR^2_{+}}))$ for any $\delta>0$. A sufficient condition ensuring this is to have a point $x^* \in \partial\bbR^2_+$ such that $\rd_{12}\tht_{0}(x^*)$ and   $\rd_{22}\tht_{0}(x^*)$ are both nonzero. 
	\end{maintheorem}
	This statement needs to be carefully distinguished with strong illposedness results in $C^{k}$ and $W^{k,\infty}$ for integer $k$ proved in \cite{CoZo,KJ-SQG} for domains without boundary: in that case, $C^\infty$ initial data lead to $C^\infty$ local in time solutions, and $C^{k}, W^{k,\infty}$ illposedness is purely coming from the initial data being exactly $C^{k}$ or $W^{k,\infty}$ and not better. On the other hand, nonexistence in Theorem \ref{thm:nonexist} is for $C^\infty$ initial data and is really coming from the boundary of the domain.

	\begin{remark}\label{rmk:DBC}
		Let us clarify that the Dirichlet boundary condition is essential for local wellposedness of strong solutions, at least when the domain is $\bbR^{2}_{+}$: it was shown in \cite{JKY} that for \textit{any} compactly supported smooth datum $\tht_{0}$ on $\bbR^{2}_{+}$ which does not satisfy \eqref{eq:DBC}, there is \textit{nonexistence} of corresponding $L^\infty([0,\dlt];C^1( \overline{\bbR^2_+} ))$ solution to \eqref{eq:SQG} with any $\dlt>0$. This is an optimal nonexistence result, since one cannot expect wellposedness of strong solutions strictly below $C^{1}$, this space being scaling critical.  
	\end{remark}

	\subsection{Discussion on singularity problem for SQG}
	
	Besides extending the results of \cite{CNgu}, another motivation of the current work is to explore settings in which finite time singularity formation for \eqref{eq:SQG} can be established. Singularity formation for \eqref{eq:SQG} is a notoriously difficult open problem, and there does not seem to be convincing numerical evidence towards blowup for smooth data in domains without boundaries. We refer to \cite{CGI,Cor-hyp,CCGM,CCZ,EJ1,HH-sqg,HMsqg,HuShuZha} for some results, against singularity formation, regarding the SQG equation. The problem is still challenging for the generalized models of the form \begin{equation}  \label{eq:alp-SQG} 
		\left\{
		\begin{aligned} 
			&\rd_t\tht + u\cdot \nb \tht = 0, \\
			&u=  \nb^\perp (-\lap)^{-1 + \frac{\alp}{2} } \tht,
		\end{aligned}
		\right. 
	\end{equation} which are sometimes referred to as the $\alp$-SQG equation. With this notation, $\alp = 1$ corresponds to \eqref{eq:SQG} and $\alp = 0$ to the vorticity formulation for two-dimensional incompressible Euler equations. A larger value of $\alp$ makes the kernel for the velocity more singular. Global regularity of smooth solutions for any $\alp>0$ in $\bbR^{2}$ is a challenging open problem. 
	
	Towards the blowup problem, the half-plane provides a nice stage, the main point being that ``patch type'' data attached to the boundary cannot be detached in finite time, unless a singularity of some form occurs. Essentially, this is the common setup for all the existing blowup results for \eqref{eq:alp-SQG}: \cite{KRYZ,KYZ,GaPa,MTXX} deal with patch solutions attached to the boundary and the very recent work \cite{Z2} concerns solutions in properly weighted H\"older spaces, non-vanishing on the boundary.

	Unfortunately, all of these results work in the range $0<\alp\le1/2$ while the SQG corresponds to $\alp=1$: there is a serious difficulty already in finding a local wellposedness class for non-vanishing data on the boundary when $1/2<\alp\le1$ (\cite{Z2,JKY}). In particular, as we remarked earlier, in the SQG case, there is nonexistence of $C^1$ solutions corresponding to a datum which is not identically zero on the boundary, even if it is $C^{\infty}$ up to the boundary and compactly supported. 
	
	Therefore, the wellposedness class that we present in Theorem \ref{thm:loc} seems to be the ``optimal'' half-plane result for SQG and could be a reasonable space to search for finite time singularity formation, the main difference with the $\bbR^{2}$ case being that the second derivative of the velocity  has a logarithmically divergent component depending on the second derivative of $\tht$ restricted to the boundary. A careful study of the second derivative of $\tht$ on the boundary might provide some scenario for finite time blowup, or at least for small scale creation (roughly speaking, a large growth of some Sobolev or H\"older norms of the solution). We remark that for the SQG equation, results on long time large growth are also rare (\cite{KN,DEJ,HeKi}), since the strongest conservation law $\nrm{\tht}_{L^{\infty}}$ is not able to control $\nrm{u}_{L^{\infty}}$.

	

	\subsection{Ideas and difficulties of the proof}
	
	In this section, we discuss several key points and difficulties involved in the proofs of our main results. 
	
	\medskip 
	
	\noindent \textbf{Proof of wellposedness}. To simplify the discussion, let us assume that $\tht$ is a given on $\bbR^{2}_{+}$ and is infinitely differentiable up to the boundary, and denote its odd extension to $\bbR^{2}$ as $\bar{\tht}$. Since the velocity is defined by $u = R^{\perp}\bar\tht$, if $\tht$ is non-vanishing on the boundary, then we see that $u_{1} = -R_{2}\bar{\tht}$ is not even locally bounded in $\bbR^{2}_{+}$. Even if we impose the Dirichlet boundary condition for $\tht$, $\rd_{22}\tht$ does not need to vanish on the boundary, and therefore $\rd_{22}u_{1} = -R_{2}(\rd_{22}\bar{\tht})$ will have a logarithmic divergent term near the boundary. The velocity does not belong to $W^{2,\infty}$, and hence it is natural to expect that the local wellposedness threshold is $W^{2,\infty}$ (or $C^{2}$) as well. Indeed, $W^{2,p}$ wellposedness for $p<\infty$ is given in \cite{CNgu} and $C^{1,\bt}$ wellposedness on $\bbR^{2}_{+}$ for any $0<\bt<1$ is straightforward, using the odd extension (which again belongs to $C^{1,\bt}$ under \eqref{eq:DBC}) and applying the existing $C^{1,\bt}$ wellposedness in $\bbR^{2}$. Therefore, the main point in Theorem \ref{thm:loc} is that wellposedness holds beyond the $C^{2}$ threshold. This relies on the specific structure of the nonlinearity $R^\perp(\bar{\tht})\cdot\nb\tht$, and roughly speaking is based on the fact that $\rd_{1}$ is paired only with $\rd_{2}$ and not again with $\rd_{1}$ in the nonlinear term. 
	
	For the $C^{2,\bt}$ local wellposedness, we need to investigate H\"older regularity for the second derivatives of the velocity (see Lemma \ref{lem:vel}). Given $\tht \in C^{2,\bt}(\ohp)$, the only component of $\nb^2 u $ which is not $C^{\bt}$ is $\rd_{22}u_{1}$, for the reason that we explained earlier. To begin with, it is straightforward to ``extract'' the logarithmically divergent component from $\rd_{22}u_{1}$, namely $4\log(x_{2})\rd_{22}\tht(x)$, which naturally appears when one tries to commute an $x_{2}$ derivative with the Riesz transform. However, somewhat surprisingly, the difference $\rd_{22}u_{1} - 4\log(x_{2})\rd_{22}\tht(x)$ still does not belong to $C^{\bt}$; the reason is that the singular integral operator defined by $\rd_{22}\tht \mapsto \rd_{22}u_{1} - 4\log(x_{2})\rd_{22}\tht(x)$ has a kernel which does not satisfy the (usual) cancellation condition on the sphere. Still, $\rd_{22}u_{1} - 4\log(x_{2})\rd_{22}\tht(x)$ is ``almost'' $C^{\bt}$, up to a logarithmic loss of the modulus of continuity, and this turns out to be sufficient for wellposedness.
	
	The proof for the $W^{3,p}$ case is more subtle, since we need to estimate third derivatives of the velocity. Commuting derivatives carefully with Riesz transforms in $L^{p}$ shows that the only problematic third order derivative of the velocity is $\rd_{222}u_{1}$, given that $\tht \in W^{3,p}(\bbR^3_+)$. Similarly in the H\"older case, we extract the ``main term'' from $\rd_{222}u_{1}$ in $L^{p}$; it turns out that it is given by \begin{equation*}
		\begin{split}
			\int_{\bbR} \frac{2x_{2}}{|x-(y_1,0)|^{3}} \rd_{22}\tht (y_1,0) \, \ud y_{1}. 
		\end{split}
	\end{equation*} Our main observation, stated in Lemma \ref{lem:vel-W3p}, is that this expression belongs to $L^{2p}(\bbR^2)$ after multiplying by another factor of $x_{2}$, when $\tht \in W^{3,p}$ for $1<p<\infty$. In the nonlinear estimate, we can gain this additional factor of $x_{2}$ from the structure of the nonlinearity: in the estimate for $\rd_{t}\rd_{222}\tht$, the term involving $\rd_{222}u_{1}$ can be written as \begin{equation*}
	\begin{split}
		\rd_{222}u_{1} \rd_{1}\tht = x_{2}\rd_{222}u_{1} \frac{\rd_{1}\tht}{x_{2}} 
	\end{split}
	\end{equation*} where $\rd_{1}\tht/x_{2}$ is also bounded in $L^{2p}$ thanks to \eqref{eq:DBC} and Hardy's inequality. 
	
	\medskip 
	
	\noindent \textbf{Proof of illposedness}. The idea behind $W^{3,\infty}$ illposedness is very simple: since we expect $\rd_{22}u_{1}(x) \sim \log(x_{2})$, the equation for  $\rd_{t}\rd_{222}\tht(x)$ has a logarithmically divergent term as well. However, it turns out to be very difficult to directly use the evolution equation for the third derivative of $\tht$: essentially, the enemy is the $W^{3,\infty}$ illposedness of \eqref{eq:SQG} in the absence of boundaries. More precisely, while we may assume the existence of a local in time $W^{3,\infty}$ solution $\tht$ (for the sake of deriving a contradiction), the corresponding velocity field does not belong to $W^{3,\infty}$ even in the absence of boundaries, due to the failure of $L^{\infty}$ bounds of the Riesz transforms. In other words, for this half-plane illposedness result, one needs to distinguish the unboundedness coming from the physical boundary with that from failure of $L^{\infty}$ bounds for the Riesz transforms. To this end, our approach is to carefully study the evolution equation for the difference quotient \begin{equation*}
		\begin{split}
			\frac{ \rd_{22}\tht(t,x_1,x_2) - \rd_{22}\tht(t,x_1,0) }{x_{2}}, 
		\end{split}
	\end{equation*} and show that it diverges like $\log x_2$.

	\subsection{Outline of the paper}
	
	The remainder of this paper is organized as follows. In Section \ref{sec:key}, we collect several key estimates for the velocity. The proof of wellposedness (Theorem \ref{thm:loc}) is given in Section \ref{sec:lwp}. Lastly, we prove nonexistence result (Theorem \ref{thm:nonexist}) in Section \ref{sec:illposed}.

	\subsection{Notation} In the proofs, the value of $c, C>0$ may change from a line to another, or even within a single line. We denote the dependency of the constant with respect to various parameters using subscripts. Given a function $f$ on $\bbR^2_+$, we denote $\overline{f}$ and $\widetilde{f}$ as its odd and even extensions in the $x_2$-variable, respectively: \begin{equation*}
		\begin{split}
			\overline{f}(\bar{x}) = \overline{f}(x_1,-x_2) = -f(x_1,x_2) \qquad \mbox{and} \qquad \widetilde{f}(\bar{x}) = \widetilde{f}(x_1,-x_2) = f(x_1,x_2)
		\end{split}
	\end{equation*}for $x_{2}>0$, where we are also using the notation $\bar{x} = (x_{1}, -x_{2})$. Lastly, throughout the paper we shall assume that the solution $\tht$ is compactly supported in space, and use $L>0$ for the length scale of the support. We write $B(0;L)$ for the ball centered at the origin with radius $L$.

	\section{Preliminary Estimates}\label{sec:key}
	
	\subsection{H\"older estimates}
	
	\begin{lemma}\label{lem:Holder}
		We have the following estimates for $0<\bt<1$. 
		\begin{enumerate}
			\item If $f \in C^\bt_0(\ohp)$ then $\nrm{\overline{f}}_{C^{\bt}(\bbR^2)} \le C_\bt \nrm{f}_{C^\bt(\ohp)}$. 
			\item If $f \in C^{1,\bt}_0(\ohp)$ then $\nrm{\overline{f}}_{C^{1,\bt}(\bbR^2)} \le C_\bt \nrm{f}_{C^{1,\bt}(\ohp)}$. 
		\end{enumerate}
	\end{lemma}
	\begin{proof}
		To prove (1), we take two points $x, x'$ which are assumed to satisfy $x \in \bbR^2_+$ and $x' \notin \bbR^2_+$. The line segment connecting $x$ to $x'$ meets $\partial\bbR^2_+$ at a unique point $z$. Then we write \begin{equation*}
			\begin{split}
				|\overline{f}(x) - \overline{f}(x')| &\le |\overline{f}(x) - \overline{f}(z)| + |\overline{f}(z) - \overline{f}(x')| \\
				& \le |f(x)-f(z)| + |f(z)-f(\overline{x'})| \le C(|x-z|^\bt + |z-\overline{x'}|^\bt) \le C(|x-x'|^\bt)
			\end{split}
		\end{equation*} where $\overline{x'}$ is the reflection point of $x'$ across the $x_1$-axis, and we have used that $f(z)=-f(z)=0$.

		To prove (2), we note that the odd extension of $\rd_1f$ to $\bbR^2$ satisfies $C^\bt$ estimate thanks to (1). To see it for $\rd_2f$, we note that $\rd_2 \overline{f}$ is the even extension of $\rd_2 f$ across $x_2$. This implies $\rd_2\overline{f} \in C^{\bt}$.

	\end{proof}


	\subsection{Sobolev estimates}
	Given a measurable function $g$ defined on the upper half-plane $\bbR^{2}_{+}$, recall that $\overline{g}$ and $\widetilde{g}$ represent their odd and even extensions to $\bbR^{2}$, respectively. 
	
	\begin{lemma}\label{lem:interchange}
		Let $g \in W^{1,p}(\bbR^{2}_{+})$ for $1 \leq p < \infty$. Then, \begin{equation*}
			\begin{split}
				\rd_{1}\overline{g} = \overline{ \rd_{1}g }, \qquad \rd_{1}\widetilde{g} = \widetilde{ \rd_{1}g }, \qquad \rd_{2} \widetilde{g} = \overline{\rd_{2}g}. 
			\end{split}
		\end{equation*} Here, the LHS is the distributional derivative in $\bbR^{2}$ of the odd (or even) extension, and the RHS is the odd (or even) extension of the distributional derivative in $\bbR^2_{+}$. Next, if we assume further that $g \equiv 0$ on $\partial\bbR^{2}_{+}$, then we have \begin{equation*}
			\begin{split}
				\rd_{2} \overline{g} = \widetilde{\rd_{2}g}. 
			\end{split}
		\end{equation*}
	\end{lemma}
	\begin{remark}\label{rmk:W1p-extension}
		Let $1 \leq p < \infty$. Then, we have that \begin{itemize}
			\item $\widetilde{g}\in W^{1,p}(\bbR^{2})$, if $g \in W^{1,p}(\bbR^{2}_{+})$
			\item  $\overline{g}\in W^{1,p}(\bbR^{2})$, if $g \in W^{1,p}_{0}(\bbR^{2}_{+})$ 
			\item $\rd_{22}\overline{g} = \overline{\rd_{22}g}$ and $\overline{g} \in W^{2,p}(\bbR^2)$, if $g \in W^{2,p}_{0}(\bbR^{2}_{+})$. 
		\end{itemize}
	\end{remark}
	\begin{proof}
		Let $g \in W^{1,p}(\bbR^{2}_{+})$, hence, \begin{equation}\label{eq:weak_der}
			\begin{aligned}
				\int_{\bbR^2_{+}} \partial_j g(x) \varphi(x) \,\ud x = - \int_{\bbR^2_{+}} g(x) \partial_j \varphi(x) \,\ud x, \qquad j = 1,\, 2
			\end{aligned}
		\end{equation} for any $\varphi \in C^{\infty}_c(\bbR^2_{+})$. We show that \begin{equation*}
			\begin{aligned}
				\int_{\bbR^2} \overline{\partial_1 g}(x) \varphi(x) \,\ud x = - \int_{\bbR^2} \overline{g}(x) \partial_1 \varphi(x) \,\ud x
			\end{aligned}
		\end{equation*} holds for all $\varphi \in C^{\infty}_c(\bbR^2)$. We consider an even function $\sigma(\cdot) \in C^{\infty}_c(\bbR)$ such that $\sigma = 1$ on $[-\frac{1}{2},\frac{1}{2}]$ and $\sigma = 0$ on $(-\infty,-1] \cup [1,\infty)$. Then, we can decompose the left-hand side as \begin{equation*}
			\begin{aligned}
				\int_{\bbR^2} \overline{\partial_1 g}(x) \sigma(\frac{x_2}{\varepsilon}) \varphi(x) \,\ud x + \int_{\bbR^2} \overline{\partial_1 g}(x) (1-\sigma(\frac{x_2}{\varepsilon})) \varphi(x) \,\ud x
			\end{aligned}
		\end{equation*} for any $\varepsilon > 0$. Note the first integral converges to $0$ as $\varepsilon \to 0$ because it is bounded by \begin{equation}\label{eq:sigma_1}
			\begin{aligned}
				2\| \varphi \|_{L^{\infty}(\bbR^2)} \int_{K_{\varepsilon}} | \nabla g(x) | \,\ud x \leq 2\| \varphi \|_{L^{\infty}(\bbR^2)} \| \nabla g \|_{L^p(K_{\varepsilon})} |K_{\varepsilon}|^{1-\frac{1}{p}},
			\end{aligned}
		\end{equation} where $K_{\varepsilon} = \{ (x_1,x_2) \in \bbR^2 ; -R < x_1 < R,\, 0 < x_2 < \varepsilon \}$ for sufficiently large $R>0$. On the othr hand, we have \begin{equation*}
			\begin{aligned}
				\int_{\bbR^2} \overline{\partial_1 g}(x) (1-\sigma(\frac{x_2}{\varepsilon})) \varphi(x) \,\ud x = \int_{\bbR^2_{+}} \partial_1 g(x) (1-\sigma(\frac{x_2}{\varepsilon})) \varphi(x) \,\ud x - \int_{\bbR^2_{+}} \partial_1 g(x) (1-\sigma(\frac{x_2}{\varepsilon})) \varphi(\overline{x}) \,\ud x,
			\end{aligned}
		\end{equation*} where $(1-\sigma(\frac{x_2}{\varepsilon})) \varphi(x)$ and $(1-\sigma(\frac{x_2}{\varepsilon})) \varphi(\overline{x})$ belong to $C^{\infty}_c(\bbR^2_{+})$. Applying \eqref{eq:weak_der} to each integral, we have \begin{equation*}
			\begin{aligned}
				\int_{\bbR^2} \overline{\partial_1 g}(x) (1-\sigma(\frac{x_2}{\varepsilon})) \varphi(x) \,\ud x &= -\int_{\bbR^2_{+}} g(x) (1-\sigma(\frac{x_2}{\varepsilon})) \partial_1 \varphi(x) \,\ud x + \int_{\bbR^2_{+}} g(x) (1-\sigma(\frac{x_2}{\varepsilon})) \partial_1 \varphi(\overline{x}) \,\ud x \\
				&= -\int_{\bbR^2} \overline{g}(x) (1-\sigma(\frac{x_2}{\varepsilon})) \partial_1 \varphi(x) \,\ud x.
			\end{aligned}
		\end{equation*} By the dominated convergence theorem, \begin{equation*}
			\begin{aligned}
				\int_{\bbR^2} \overline{g}(x) (1-\sigma(\frac{x_2}{\varepsilon})) \partial_1 \varphi(x) \,\ud x \to \int_{\bbR^2} \overline{g}(x) \partial_1 \varphi(x) \,\ud x
			\end{aligned}
		\end{equation*} as $\varepsilon \to 0$. Thus, $\rd_{1}\overline{g} = \overline{ \rd_{1}g }$ follows from the above restults. One may prove $\rd_{1}\widetilde{g} = \widetilde{ \rd_{1}g }$ similarly. 
		
		Next, we show that $\rd_{2} \widetilde{g} = \overline{\rd_{2}g}$, i.e., \begin{equation*}
			\begin{aligned}
				\int_{\bbR^2} \overline{\partial_2 g}(x) \varphi(x) \,\ud x = - \int_{\bbR^2} \widetilde{g}(x) \partial_2 \varphi(x) \,\ud x
			\end{aligned}
		\end{equation*} holds for all $\varphi \in C^{\infty}_c(\bbR^2)$. Following the previous procedure, we have \begin{equation*}
			\begin{aligned}
				\int_{\bbR^2} \overline{\partial_2 g}(x) \varphi(x) \,\ud x = \int_{\bbR^2} \overline{\partial_2 g}(x) \sigma(\frac{x_2}{\varepsilon}) \varphi(x) \,\ud x + \int_{\bbR^2} \overline{\partial_2 g}(x) (1-\sigma(\frac{x_2}{\varepsilon})) \varphi(x) \,\ud x,
			\end{aligned}
		\end{equation*}  regardless of the choice of $\varepsilon > 0$, and \begin{equation*}
			\begin{aligned}
				\left| \int_{\bbR^2} \overline{\partial_2 g}(x) \sigma(\frac{x_2}{\varepsilon}) \varphi(x) \,\ud x \right| \to 0
			\end{aligned}
		\end{equation*} as $\varepsilon \to 0$ due to \eqref{eq:sigma_1}. Using \eqref{eq:weak_der}, we have  \begin{equation*}
		\begin{split}
				\int_{\bbR^2} \overline{\partial_2 g}(x) (1-\sigma(\frac{x_2}{\varepsilon})) \varphi(x) \,\ud x &= -\int_{\bbR^2_{+}} g(x) \partial_2 [(1-\sigma(\frac{x_2}{\varepsilon})) \varphi(x)] \,\ud x + \int_{\bbR^2_{+}} g(x) \partial_2 [(1-\sigma(\frac{x_2}{\varepsilon})) \varphi(\overline{x})] \,\ud x \\
			&= - \int_{\bbR^2} \widetilde{g}(x) (1-\sigma(\frac{x_2}{\varepsilon})) \partial_2 \varphi(x) \,\ud x + \frac{1}{\varepsilon}\int_{\bbR^2_{+}} g(x) \sigma'(\frac{x_2}{\varepsilon}) (\varphi(x) - \varphi(\overline{x})) \,\ud x.
		\end{split}
		\end{equation*} Note from the mean value theorem that \begin{equation}\label{eq:sigma_2}
			\begin{aligned}
				\left| \frac{1}{\varepsilon}\int_{\bbR^2_{+}} g(x) \sigma'(\frac{x_2}{\varepsilon}) (\varphi(x) - \varphi(\overline{x})) \,\ud x \right| &\leq C\| \sigma' \|_{L^{\infty}(\bbR)} \int_{K_{\varepsilon}} |g(x)| \frac{|\varphi(x) - \varphi(\overline{x})|}{2x_2} \,\ud x \\
				&\leq C\| \sigma' \|_{L^{\infty}(\bbR)} \| \rd_{2} \varphi \|_{L^{\infty}(\bbR^2)} \| g \|_{L^p(K_{\varepsilon})} |K_{\varepsilon}|^{1-\frac{1}{p}}.
			\end{aligned}
		\end{equation} Consequently, we have with the dominated convergence theorem that \begin{equation*}
			\begin{aligned}
				\int_{\bbR^2} \overline{\partial_2 g}(x) (1-\sigma(\frac{x_2}{\varepsilon})) \varphi(x) \,\ud x \to -\int_{\bbR^2} \widetilde{g}(x) \partial_2 \varphi(x) \,\ud x
			\end{aligned}
		\end{equation*} as $\varepsilon \to 0$. Thus, $\rd_{2} \widetilde{g} = \overline{\rd_{2}g}$ holds.
		
		Now, we consider $g \in W^{1,p}_{0}(\bbR^{2}_{+})$. Since $C^{\infty}_c(\bbR^2_{+})$ is dense in $W^{1,p}_{0}(\bbR^2_{+})$ and $h \in C^{\infty}_c(\bbR^2_{+})$ satisfies \begin{equation*}
			\begin{aligned}
				\int_{\bbR^2} \widetilde{\partial_2 h}(x) \varphi(x) \,\ud x = - \int_{\bbR^2} \overline{h}(x) \partial_2 \varphi(x) \,\ud x
			\end{aligned}
		\end{equation*} for all $\varphi \in C^{\infty}_c(\bbR^2)$, it follows $\partial_2 \overline{g} = \widetilde{\partial_2 g}$.
		This completes the proof.
	\end{proof}

	\subsection{Key estimates on the velocity }
	
	In this subsection, we collect several key estimates on the velocity. We present three lemmas, which correspond to H\"older, log-Lipschitz, and Sobolev estimates for the velocity. 
	

	
	\begin{lemma}[H\"older estimates]\label{lem:vel}
		Fix some $0<\bt<1$ and let $\tht \in C^{2,\bt}_{0}(\ohp)$ be supported in $B(0;L)$. Then, the velocity $u =- \nb^\perp (-\lap)^{-\frac12}\tht$ satisfies \begin{equation}\label{eq:u-C2}
			\begin{split}
				\nrm{u_2}_{C^{2,\bt}(\ohp)} + \nrm{\rd_1 u_1}_{C^{1,\bt}(\ohp)} \le  C_{L}\nrm{\tht}_{C^{2,\bt}(\ohp)},
			\end{split}
		\end{equation}  \begin{equation}\label{eq:u-log1}
			\begin{split}
				\nrm{\rd_2^2 u_1(x) - 4\rd_2^2 \tht(x) \log (x_2) }_{L^{\infty}} \le  C_{L}\nrm{\tht}_{C^{2,\bt}(\ohp)},
			\end{split}
		\end{equation}  and \begin{multline}\label{eq:u-log}
			\bigg|(\rd_{2}^2u_{1} (x) - 4\rd_{2}^2 \tht(x) \log (x_2)) - (\rd_{2}^2u_1(x') - 4 \rd_{2}^2\tht(x') \log (x_2')) \\
			+ 4(\rd_{2}^2 \tht(x) - \rd_{2}^2 \tht(x')) \log (2|x-x'| + \sqrt{4|x-x'|^2 + x_2^2}) \bigg| \le C_{L}\nrm{\tht}_{C^{2,\bt}(\ohp)} |x-x'|^{\beta}
		\end{multline} for all $x$, $x' \in \bbR^2_{+}$ with $|x_1|$, $|x_1'| < 3L$. 
	\end{lemma}
	\begin{remark}\label{rmk:gm}
		Note that \eqref{eq:u-C2} and \eqref{eq:u-log1} imply $\| \partial_2 u_1 \|_{C^{\gamma}} \leq C_{L}\nrm{\tht}_{C^{2,\bt}(\ohp)}$ for any $0 < \gamma < 1$.
	\end{remark}
	\begin{remark}
		If $x \in \bbR^2_{+}$ satisfies $|x_1|> 2L$, then $\| \partial_2 u_1 \|_{W^{k,\infty}(\bbR^2_{+})} \leq C_{L,k} \| \rd_{2}^2 \tht \|_{L^{\infty}}$ for all $k \geq 0$. If $x$, $x' \in \bbR^2_{+}$ satisfy $|x_1| < 2L < 3L < |x_1'|$, then one have $|x_1 - x_1'| \geq L$, and  \begin{equation*}
			\begin{split}
				|\rd_{2}^2u_{1} (x) - 4\rd_{2}^2 \tht(x) \log (x_2) - (\rd_{2}^2u_1(x') - 4 \rd_{2}^2\tht(x') \log (x_2')) | \leq C_{L}\nrm{\tht}_{C^{2,\bt}(\ohp)} |x-x'|.
			\end{split}
		\end{equation*}
	\end{remark}
	
	\begin{proof} We only provide a proof for $u_1$ since one can adjust it to $u_2$ and complete the proof without difficulty. We recall the formula \begin{equation*}
			\begin{split}
				u_{1} (x) = P.V. \int_{\bbR_+^2} \left[ \frac{x_2-y_2}{|x-y|^3} -  \frac{x_2 + y_2 }{|x-\overline{y}|^3} \right] \tht(y) \,\ud y = P.V. \int_{\bbR^2} \frac{x_2-y_2}{|x-y|^3} \overline{\tht(y)} \,\ud y.
			\end{split}
		\end{equation*} The odd extension of $\tht$ satisfies $\overline{\tht} \in C^1_c(\mathbb{R}^2)$, $\partial_1 \overline{\tht} \in C^{1,\beta}$, and $\partial_2 \overline{\tht} \in \operatorname{Lip}(\mathbb{R}^2)$. Note that $\partial_2^2 \overline{\tht}$ could be discontinuous at the boundary when $\lim_{y_2 \to 0^+} \partial_2^2\theta(y) \neq 0$ for some $y_1 \in \bbR$. Thus, we have $\partial_1 u_1 \in C^{1,\beta}(\mathbb{R}^2)$ since for any $f \in C^{1,\bt}_c(\mathbb{R}^2)$ \begin{equation*}
			\begin{split}
				\partial_i \left( P.V. \int_{\bbR^2} \frac{(x-y)^{\perp}}{|x-y|^3} f(y) \,\ud y \right) = P.V. \int_{\bbR^2} \frac{(x-y)^{\perp}}{|x-y|^3} \partial_i f(y) \,\ud y \in C^{\bt}(\mathbb{R}^2)
			\end{split}
		\end{equation*} holds. To estimate $\partial_2^2 u_1$, we note \begin{equation*}
			\begin{split}
				\partial_2 u_1(x) = P.V. \int_{\bbR^2} \frac{x_2-y_2}{|x-y|^3} \partial_2 \overline{\tht(y)} \,\ud y \in C^{\gamma}(\mathbb{R}^2), \qquad \gamma < 1.
			\end{split}
		\end{equation*} We split the integral as \begin{equation*}
			\begin{gathered}
				P.V. \int_{\bbR^2_{+}} \frac{x_2-y_2}{|x-y|^3} \partial_2 \tht(y) \,\ud y + P.V. \int_{\bbR^2_{+}} \frac{x_2+y_2}{|x-\overline{y}|^3} \partial_2 \tht(y) \,\ud y .
			\end{gathered}
		\end{equation*} Using integration by parts with \begin{equation*}
			\begin{split}
				\rd_{y_2} \left( \frac{1}{|x-y|} \right) = \frac{x_2-y_2}{|x-y|^3}, \qquad \rd_{y_2} \left( \frac{1}{|x-\overline{y}|} \right) = -\frac{x_2+y_2}{|x-\overline{y}|^3},
			\end{split}
		\end{equation*} we arrive at \begin{equation*}
			\begin{split}
				- P.V. \int_{\bbR^2_{+}} \frac{1}{|x-y|} \partial_2^2 \tht(y) \,\ud y + P.V. \int_{\bbR^2_{+}} \frac{1}{|x-\overline{y}|} \partial_2^2 \tht(y) \,\ud y = - \int_{\bbR^2_{+}} \left[ \frac{1}{|x-y|} - \frac{1}{|x-\overline{y}|} \right] \partial_2^2 \tht(y) \,\ud y,
			\end{split}
		\end{equation*} where the boundary terms cancel each other out. We first split $\partial_2^2 \tht(y)$ as $\partial_2^2 \tht(y_1,0) + (\partial_2^2 \tht(y)- \partial_2^2 \tht(y_1,0))$, and note that the corresponding integral for $\partial_2^2 \tht(y)- \partial_2^2 \tht(y_1,0)$ belongs to $C^{1,\bt}(\bbR^2)$ simply because the odd extension of $\partial_2^2\tht(y)-\partial_2^2 \tht(y_1,0)$ belongs to $C^\bt(\bbR^2)$. After using \begin{equation*}
			\begin{split}
				\partial_\tau \int_{0}^{\delta} f(\tau-\eta) \,\ud \eta = f(\tau) - f(\tau - \delta), \qquad f \in C(\bbR), \qquad \tau,\, \delta \in \bbR,
			\end{split}
		\end{equation*} it only remains to consider the part  \begin{equation}\label{eq:log_term}
			\begin{split}
				&- \partial_2 \int_{\bbR} \left[ \int_{y_2 \ge 0} \frac{1}{|x-y|} \,\ud y_2 \right] \partial_2^2 \tht(y_1,0) \,\ud y_1 + \partial_2 \int_{\bbR} \left[ \int_{y_2 \ge 0} \frac{1}{|x-\overline{y}|} \,\ud y_2 \right] \partial_2^2 \tht(y_1,0) \,\ud y_1 \\
				&\qquad= -2 \int_{\bbR} \frac{1}{ \sqrt{ (x_1-y_1)^2 + x_2^2 }  } \rd_{2}^2 \tht(y_1,0) \, \ud y_1.
			\end{split}
		\end{equation}  We prove \eqref{eq:u-log1} first. If $|x| \le 3L$, the domain of $y_1$ integral can be restricted to $|x_1-y_1| \le 4L$. (Estimating the integral is simpler if $|x|>3L$.) We further decompose the last integral as \begin{equation*}
			\begin{split}
				-2 \int_{|x_1-y_1| \le 4L} \frac{1}{ \sqrt{ (x_1-y_1)^2 + x_2^2 }  } \left(  (\rd_{2}^2\tht(y_1,0) - \rd_{2}^2\tht(x)) + \rd_{2}^2\tht(x) \right) \, \ud y_1 
			\end{split}
		\end{equation*} The contribution to the integral from $(\rd_{2}^2\tht(y_1,0) - \rd_{2}^2\tht(x))$ is uniformly bounded. We note that \begin{equation}\label{eq:log_x2}
			\begin{split}
				\int_{|x_1-y_1| \le 4L } 	\frac{1}{ \sqrt{ (x_1-y_1)^2 + x_2^2 }  } \, \ud y_1 + 2\log(x_2) = 2\log (4L + \sqrt{16L^2+x_2^2}),
			\end{split}
		\end{equation} thus, \eqref{eq:u-log1} follows. To prove \eqref{eq:u-log}, fix $x$ and $x'$ with $|x_1|$, $|x_1'| \leq 3L$. Thanks to \eqref{eq:log_term}, it suffices to show that \begin{gather*}
			\bigg| -2 \int_{|x_1-y_1| \le 10L } \frac{1}{ \sqrt{ (x_1-y_1)^2 + x_2^2 }  } \rd_{2}^2 \tht(y_1,0) \, \ud y_1 +2 \int_{|x_1-y_1| \le 10L } \frac{1}{ \sqrt{ (x_1'-y_1)^2 + x_2'^2 }  } \rd_{2}^2 \tht(y_1,0) \, \ud y_1 \\
			- 4 \rd_{2}^2 \tht(x) \log (x_2) + 4 \rd_{2}^2 \tht(x') \log (x_2') + 4 (\rd_{2}^2 \tht(x) - \rd_{2}^2 \tht(x')) \log (2|x-x'| + \sqrt{4|x-x'|^2 + x_2^2}) \bigg| \\
			\leq C \| \tht \|_{C^{2,\beta}(\ohp)} |x-x'|^{\beta}.
		\end{gather*} For this, we let \begin{gather*}
			-2 \int_{|x_1-y_1| \le 10L } \frac{1}{ \sqrt{ (x_1-y_1)^2 + x_2^2 }  } \rd_{2}^2 \tht(y_1,0) \, \ud y_1 +2 \int_{|x_1-y_1| \le 10L } \frac{1}{ \sqrt{ (x_1'-y_1)^2 + x_2'^2 }  } \rd_{2}^2 \tht(y_1,0) \, \ud y_1 \\
			=J_1+J_2+J_3,
		\end{gather*} where \begin{equation*}
			\begin{split}
				J_1 := &-2 \int_{|x_1-y_1| \le 2|x-x'| } \frac{1}{ \sqrt{ (x_1-y_1)^2 + x_2^2 }  } (\rd_{2}^2 \tht(y_1,0) - \rd_{2}^2 \tht(x)) \, \ud y_1 \\
				&+2 \int_{|x_1-y_1| \le 2|x-x'| } \frac{1}{ \sqrt{ (x_1'-y_1)^2 + x_2'^2 }  } (\rd_{2}^2 \tht(y_1,0) - \rd_{2}^2 \tht(x')) \, \ud y_1,
			\end{split}
		\end{equation*} \begin{equation*}
			\begin{split}
				J_2 := &-2 \int_{2|x-x'| \le |x_1-y_1| \le 10L } \frac{1}{ \sqrt{ (x_1-y_1)^2 + x_2^2 }  } (\rd_{2}^2 \tht(y_1,0) - \rd_{2}^2 \tht(x)) \, \ud y_1 \\
				&+2 \int_{2|x-x'| \le |x_1-y_1| \le 10L } \frac{1}{ \sqrt{ (x_1'-y_1)^2 + x_2'^2 }  } (\rd_{2}^2 \tht(y_1,0) - \rd_{2}^2 \tht(x')) \, \ud y_1,
			\end{split}
		\end{equation*} \begin{equation*}
			J_3 := -2 \int_{|x_1-y_1| \le 10L } \frac{1}{ \sqrt{ (x_1-y_1)^2 + x_2^2 }  } \rd_{2}^2 \tht(x) \, \ud y_1 +2 \int_{|x_1-y_1| \le 10L } \frac{1}{ \sqrt{ (x_1'-y_1)^2 + x_2'^2 }  } \rd_{2}^2 \tht(x') \, \ud y_1,
		\end{equation*} and prove \begin{equation}\label{eq:claim1}
			|J_1| \leq C\| \rd_{2}^2 \tht \|_{C^{\beta}(\ohp)}|x-x'|^{\beta},
		\end{equation} \begin{equation}\label{eq:claim2}
			|J_2 + 4 (\rd_{2}^2 \tht(x) - \rd_{2}^2 \tht(x')) \log (2|x-x'| + \sqrt{4|x-x'|^2 + x_2^2})| \leq C \| \rd_{2}^2 \tht \|_{C^{\beta}(\ohp)} |x-x'|^{\beta},
		\end{equation} \begin{equation}\label{eq:claim3}
			|J_3 - 4 \rd_{2}^2 \tht(x) \log (x_2) + 4 \rd_{2}^2 \tht(x') \log (x_2')| \leq C \| \rd_{2}^2\tht \|_{C^{\beta}(\ohp)} |x-x'|^{\beta}.
		\end{equation} Since \eqref{eq:claim1} is clear, we omit the proof for it. $J_2$ corresponds to \begin{equation*}
			\begin{gathered}
				-2 \int_{2|x-x'| \leq |x_1-y_1| \leq 10L} \left( \frac{1}{ \sqrt{ (x_1-y_1)^2 + x_2^2 }  } - \frac{1}{ \sqrt{ (x_1'-y_1)^2 + x_2'^2 }  } \right) (\rd_{2}^2 \tht(y_1,0) - \rd_{2}^2 \tht(x')) \, \ud y_1 \\
				+2 \int_{2|x-x'| \leq |x_1-y_1| \leq 10L} \frac{1}{ \sqrt{ (x_1-y_1)^2 + x_2^2 }  } (\rd_{2}^2 \tht(x) - \rd_{2}^2 \tht(x')) \, \ud y_1.
			\end{gathered}
		\end{equation*} Note from \begin{equation*}
			\begin{gathered}
				\frac{1}{ \sqrt{ (x_1-y_1)^2 + x_2^2 }  } - \frac{1}{ \sqrt{ (x_1'-y_1)^2 + x_2'^2 }  } 	\leq \frac{|x_1 - x_1'| + |x_2 - x_2'|}{ \sqrt{ (x_1'-y_1)^2 + x_2'^2 } \sqrt{ (x_1-y_1)^2 + x_2^2 } }
			\end{gathered}
		\end{equation*} that \begin{equation*}
			\begin{gathered}
				\left| -2 \int_{|x_1-y_1| \le 10L } \left( \frac{1}{ \sqrt{ (x_1-y_1)^2 + x_2^2 }  } - \frac{1}{ \sqrt{ (x_1'-y_1)^2 + x_2'^2 }  } \right) (\rd_{2}^2 \tht(y_1,0) - \rd_{2}^2 \tht(x')) \, \ud y_1 \right| \leq C \| \tht \|_{C^{\beta}(\ohp)} |x-x'|^{\beta}.
			\end{gathered}
		\end{equation*} On the other hand, we can directly compute \begin{equation*}
			\begin{gathered}
				2 \int_{2|x-x'| \leq |x_1-y_1| \leq 10L} \frac{1}{ \sqrt{ (x_1-y_1)^2 + x_2^2 }  } (\rd_{2}^2 \tht(x) - \rd_{2}^2 \tht(x')) \, \ud y_1 \\
				= 4 (\rd_{2}^2 \tht(x) - \rd_{2}^2 \tht(x')) \left( \log ( 10L + \sqrt{100L^2 + x_2^2} ) - \log (2|x-x'| + \sqrt{4|x-x'|^2 + x_2^2}) \right).
			\end{gathered}
		\end{equation*} Thus, \eqref{eq:claim2} follows. To estimate $J_3$, we see from the first integral of $J_3$ that \begin{equation}\label{eq:log_sing}
			-2 \int_{|x_1-y_1| \le 10L } \frac{1}{ \sqrt{ (x_1-y_1)^2 + x_2^2 }  } \rd_{2}^2 \tht(x) \, \ud y_1 - 4 \rd_{2}^2 \tht(x) \log (x_2) = -4 \rd_{2}^2 \tht(x) \log (10L + \sqrt{100L^2 + x_2^2}).
		\end{equation} The second integral of $J_3$ is equal to \begin{equation*}
			\begin{gathered}
				2 \int_{|x_1-y_1| \le 10L } \frac{1}{ \sqrt{ (x_1'-y_1)^2 + x_2'^2 }  } \rd_{2}^2 \tht(x') \, \ud y_1 - 2 \int_{|x_1'-y_1| \le 10L } \frac{1}{ \sqrt{ (x_1'-y_1)^2 + x_2'^2 }  } \rd_{2}^2 \tht(x') \, \ud y_1 \\
				= 2 \int_{10L + x_1'}^{10L + x_1} \frac{1}{ \sqrt{ (x_1'-y_1)^2 + x_2'^2 }  } \rd_{2}^2 \tht(x') \, \ud y_1 + 2 \int_{-10L + x_1}^{-10L + x_1'} \frac{1}{ \sqrt{ (x_1'-y_1)^2 + x_2'^2 }  } \rd_{2}^2 \tht(x') \, \ud y_1, 
			\end{gathered}
		\end{equation*} where \begin{equation*}
			2 \int_{|x_1'-y_1| \le 10L } \frac{1}{ \sqrt{ (x_1'-y_1)^2 + x_2'^2 }  } \rd_{2}^2 \tht(x') \, \ud y_1 + 4 \rd_{2}^2 \tht(x') \log (x_2') = 4 \rd_{2}^2 \tht(x') \log (10L + \sqrt{100L^2 + x_2'^2})
		\end{equation*}
		holds by \eqref{eq:log_sing} and
		the right-hand side is bounded by $C\| \rd_{2}^2 \tht \|_{L^{\infty}(\ohp)}|x_1-x_1'|^{\beta}$. Combining these results with \begin{equation*}
			\left| -4 \rd_{2}^2 \tht(x) \log (10L + \sqrt{100L^2 + x_2^2}) + 4 \rd_{2}^2 \tht(x') \log (10L + \sqrt{100L^2 + x_2'^2}) \right| \leq C \| \tht \|_{C^{\beta}(\ohp)} |x-x'|^{\beta},
		\end{equation*} we deduce \eqref{eq:claim3}. This completes the proof.
	\end{proof}

	\begin{lemma}[Log-Lipschitz estimate]
		Let $\tht \in C^{2}_{0}(\ohp)$ be supported in $B(0;L)$  and $\nabla^2 \tht \in \operatorname{Lip}(\bbR^2_{+})$. Then, the corresponding velocity  satisfies \begin{equation}\label{eq:log-Lipschitz}
			\begin{split}
				|\nabla \partial_1 u(x) - \nabla \partial_1 u(x')| \leq C_L \| \nabla \partial_1 \tht \|_{Lip} |x-x'| \log \left( 10 +\frac{1}{|x-x'|} \right).
			\end{split}
		\end{equation} 
		for all $x$, $x' \in \bbR^2_{+}$ with $|x_1|$, $|x_1'| < 3L$. 
	\end{lemma}
	\begin{proof}
		We recall that \begin{equation*}
			\begin{split}
				\partial_1^2 u = P.V. \int_{\bbR^2} \frac{(x-y)^{\perp}}{|x-y|^3} \partial_1^2 \overline{\tht(y)} \,\ud y \qquad\mbox{and}\qquad 				\partial_{12} u = P.V. \int_{\bbR^2} \frac{(x-y)^{\perp}}{|x-y|^3} \partial_{12} \overline{\tht(y)} \,\ud y 
			\end{split}
		\end{equation*}   where $\overline{\tht}$ is the odd extension of $\tht$ in $x_2$ such that $\partial_1^2 \overline{\tht}$, $\partial_{12} \overline{\tht} \in \operatorname{Lip}(\bbR^2)$. Therefore, one can prove in the standard way that \eqref{eq:log-Lipschitz} holds. 
	\end{proof}
	
	\begin{lemma}[Sobolev regularity of the velocity]\label{lem:vel-W3p}
		Let $\tht \in W^{3,p}_0(\bbR^2_{+})$ for some $1 < p < \infty$. Then, the corresponding velocity  satisfies \begin{equation}\label{eq:vel-reg}
			\begin{aligned}
				\| \nabla u \|_{L^{\infty}(\bbR^2)} + \| \nabla^2 u \|_{L^{2p}(\bbR^2)} + \| \nabla^2 \partial_1 u \|_{L^{p}(\bbR^{2})} \leq C \| \tht \|_{W^{3,p}(\bbR^2_{+})},
			\end{aligned}
		\end{equation} and \begin{equation}\label{eq:vel-sing}
			\begin{split}
				\| \rd_{222}u_1 - 2H[\tht] \|_{L^{p}(\bbR^{2})} \leq C\| \rd_{222} \tht \|_{L^p(\bbR^2_{+})},
			\end{split}
		\end{equation} where $H[\tht]$ is defined by \begin{equation}\label{eq:def-H}
			\begin{aligned}
				H[\tht](x) := \int_{\bbR} \frac{x_2}{|x-(y_1,0)|^3} \partial_{22} \tht(y_1,0) \,\ud y_1.
			\end{aligned}
		\end{equation} Furthermore, it holds that \begin{equation}\label{eq:H2p}
			\begin{aligned}
				\| x_2 H[\tht] \|_{L^{2p}(\bbR^2)}  \leq C \| \partial_{22} \tht \|_{W^{1,p}(\bbR^2_{+})}.
			\end{aligned}
		\end{equation}
	\end{lemma}

	\begin{proof}
		We fix some $1<p<\infty$ and take $\theta \in W^{3,p}_0(\bbR^2_{+})$. To begin with, from Sobolev embedding we have $\nrm{\nb u}_{L^\infty(\bbR^2)} \le C\nrm{\nb u}_{W^{1,2p}(\bbR^2)}$ and let us first show that \begin{equation}\label{eq:L2p}
			\begin{aligned}
				\| \nabla u \|_{W^{1,2p}(\bbR^2)} \leq C \| \tht \|_{W^{2,2p}(\bbR^2_{+})}. 
			\end{aligned}
		\end{equation} The claimed bound for $\nb u, \nb^2 u$ in \eqref{eq:vel-reg} then follows from applying the embedding $W^{3,p}(\bbR^2_{+}) \hookrightarrow W^{2,2p}(\bbR^2_{+})$ to $\tht$. We shall only estimate $\nrm{\nb^2 u}_{L^{2p}(\bbR^2)}$, since $\nrm{\nb u}_{L^{2p}(\bbR^2)}$ is only simpler. Using Lemma~\ref{lem:interchange} and Remark~\ref{rmk:W1p-extension}  with that $\overline{\tht}(x_1,0) = 0$ for all $x_1 \in \bbR$, we have \begin{equation*}
			\begin{aligned}
				\rd_{11}\overline{\tht} = \overline{\rd_{11}\tht}, \qquad \rd_{12}\overline{\tht} = \widetilde{\rd_{12}\tht}
			\end{aligned}
		\end{equation*} which belongs to $L^{2p}(\bbR^2)$. Hence, \begin{equation*}
			\begin{aligned}
				\| \nabla \rd_{1} u \|_{L^{2p}(\bbR^2)} = \| R^{\perp} [\nabla \rd_{1} \overline{\tht}] \|_{L^{2p}(\bbR^2)} \leq C (\| \overline{\rd_{11} \tht} \|_{L^{2p}(\bbR^2)} + \| \widetilde{\rd_{12} \tht} \|_{L^{2p}(\bbR^2)}) \leq C \| \nabla \rd_{1} \tht \|_{L^{2p}(\bbR^2_{+})}.
			\end{aligned}
		\end{equation*} To estimate $\rd_{22}u_1$, we recall from Remark~\ref{rmk:W1p-extension} that $\rd_{22}\overline{\tht} = \overline{\rd_{22}\tht} \in L^{2p}(\bbR^2),$ and then $\rd_{22} u_{1} = -R_{2} [\overline{\rd_{22}\tht}] $. Using the $L^{2p}$ boundness of $R_{2}$ in $\bbR^{2}$, we get $\| \rd_{22} u_{1} \|_{L^{2p}(\bbR^2)} \leq C \| \rd_{22} \tht \|_{L^{2p}(\bbR^2_{+})}$, which gives \eqref{eq:L2p}.

		We now estimate the third derivatives of $u$. Since $\tht \in W^{3,p}_{0}(\bbR^2_{+})$ and Remark~\ref{rmk:W1p-extension} implies \begin{equation*}
			\begin{split}
				\rd_{1}^{3} \overline{\tht} = \overline{\rd_{1}^{3}\tht} \in L^p(\bbR^2), \qquad \rd_{112} \overline{\tht} = \widetilde{\rd_{112}\tht} \in L^p(\bbR^2), \qquad \rd_{122} \overline{\tht} = \overline{\rd_{122}\tht} \in L^p(\bbR^2),
			\end{split}
		\end{equation*} we have \begin{equation*}
			\begin{split}
				\| \nabla^2 \rd_{1} u \|_{L^p(\bbR^{2})} \leq C \| \nabla^2 \rd_{1} \tht \|_{L^p(\bbR^{2}_{+})}. 
			\end{split}
		\end{equation*} This finishes the proof of \eqref{eq:vel-reg}.
		
		It only remains to estimate $\rd_{222}u_{1}$: to begin with, we decompose  $	\rd_{22} u_{1}$ as follows:  \begin{equation}\label{eq:22u1}
			\begin{aligned}
				\rd_{22} u_{1} &= -R_{2} [\overline{\rd_{22}\tht}] = - R_{2} [\widetilde{\rd_{22}\tht}] - 2 R_{2} [ \mathbf{1}_{\bbR^{2}_{-}} \overline{\rd_{22}\tht} ] \\
				&= - R_{2} [\widetilde{\rd_{22}\tht}] +2\int_{\bbR^2_{-}} \frac{x_2-y_2}{|x-y|^3} \partial_{22} \tht(\bar{y}) \,\ud y.
			\end{aligned}
		\end{equation}  Integrating by parts, we obtain\begin{equation*}
			\begin{aligned}
				2 \int_{\bbR^2_{-}} \frac{x_2-y_2}{|x-y|^3} \partial_{22} \tht(\bar{y}) \,\ud y &= 2\int_{\bbR^2_{-}} \frac{1}{|x-y|} \partial_{222} \tht(\bar{y}) \,\ud y - 2\int_{\bbR} \frac{1}{|x-(y_1,0)|} \partial_{22} \tht(y_1,0) \,\ud y_1,
			\end{aligned}
		\end{equation*} and taking the distributional derivative in $x_{2}$ to \eqref{eq:22u1} gives that  \begin{equation}\label{eq:222u1}
			\begin{aligned}
				\partial_{222}u_1 &= - R_{2} [\overline{\partial_{222}\tht}] - 2\int_{\bbR^2_{-}} \frac{x_2-y_2}{|x-y|^3} \partial_{222} \tht(\bar{y}) \,\ud y + 2H[\tht] \\
				&= - R_{2} [\overline{\partial_{222}\tht}] +2R_{2} [ \mathbf{1}_{\bbR^{2}_{-}} \overline{\rd_{222}\tht} ] + 2 H[\tht] \\
				&= - R_{2} [\widetilde{\partial_{222}\tht}] + 2 H[\tht],
			\end{aligned}
		\end{equation} where $H[\tht]$ is defined by \eqref{eq:def-H}. Hence, we have \begin{equation*}
			\| \partial_{222}u_1 - 2 H[\tht] \|_{L^p(\bbR^2)} \leq C\| R_{2} \widetilde{\partial_{222}\tht} \|_{L^p(\bbR^2)} \leq C\| \partial_{222}\tht \|_{L^p(\bbR^2_{+})}.
		\end{equation*} 
			We now estimate $\| x_2 H[\tht] \|_{L^{2p}(\bbR^2)}$. For simplicity, we introduce \begin{equation*}
				\begin{split}
					F(x_1,x_2,y_1) := \frac{x_2^{2}}{|x-(y_1,0)|^3} |\partial_{22} \tht(y_1,0)|
				\end{split}
			\end{equation*} so that \begin{equation*}
				\begin{split}
					\| x_2 H[\tht] \|_{L^{2p}(\bbR^2)} \le \left\| \nrm{F}_{L^{1}_{y_1}} \right\|_{L^{2p}_{x_2,x_1}} \le \left\| \nrm{ \nrm{F}_{L^{2p}_{x_2}} }_{L^{1}_{y_1}} \right\|_{L^{2p}_{x_1}}
				\end{split}
			\end{equation*} by Minkowski's inequality. We can explicitly compute $\nrm{F}_{L^{2p}_{x_2}}$, by a change of variables: \begin{equation*}
				\begin{split}
					\nrm{F}_{L^{2p}_{x_2}} = \frac{ c_{0} }{|x_1 - y_1|^{1 - \frac{1}{2p}}}  |\partial_{22} \tht(y_1,0)| 
				\end{split}
			\end{equation*} for some absolute constant $c_{0}$. Therefore, by the Hardy--Littlewood--Sobolev lemma and the trace inequality,  \begin{equation*}
				\begin{aligned}
					\| x_2 H[\tht] \|_{L^{2p}(\bbR^2)} & \leq C \left\| \int_{\bbR} \frac{1}{|x_1 - y_1|^{1 - \frac{1}{2p}}}  |\partial_{22} \tht(y_1,0)| \,\ud y_1 \right\|_{L^{2p}(\bbR)} \\
					&\leq C \| \partial_{22} \tht(\cdot,0) \|_{L^p(\bbR)}  \leq C\| \partial_{22} \tht \|_{W^{1,p}(\bbR^2_{+})},
				\end{aligned}
			\end{equation*} from which \eqref{eq:H2p} follows. This completes the proof.
		\end{proof}
		
		\section{High regularity wellposedness}\label{sec:lwp}

			\subsection{Preliminary remarks}\label{subsec:wellposed-rmk}
			
			We begin with some remarks on local wellposedness of \eqref{eq:SQG} on $\bbR^{2}_{+}$. For initial data $\theta_{0}$ belonging to either $C^{2,\bt}_{0}$ or $W^{3,p}_{0}$, we already have local wellposedness (existence and uniqueness) in the space $C^{1,\bt}_{0}$ or $W^{2,2p}_{0}$, respectively for any $0<\bt<1$ and $1<p<\infty$. That is, there exist $T>0$ and $L>0$ depending on $\tht_{0}$ such that we have the unique corresponding solution $\tht \in L^{\infty}([0,T]; C^{1,\bt})$ supported in the ball $B(0;L)$. Therefore, for local wellposedness in $C^{2,\bt}$, it suffices to show that this unique solution $\tht$ further belongs to $C^{2,\bt}$. A similar remark holds for the case of $W^{3,p}$ as well.

		\subsection{Propagation of $C^{2,\bt}$}\label{sec:C2-bt}
		
		In this section, we prove local wellposedness of \eqref{eq:SQG} in $C^{2,\beta}_0$ for any $0 < \beta < 1$. 
		
		\subsubsection{A priori estimate}
		
		We fix $0 < \beta < 1$ and let $\tht$ be a $L^{\infty}([0,T];C^{2,\beta}_0)$--solution to \eqref{eq:SQG} for some $T>0$ satisfying $\supp \tht(t) \subset B(0;L)$ for some $L>0$. We can apply Lemma~\ref{lem:vel} and have \eqref{eq:u-C2}, in particular, $u \in \operatorname{Lip}(\mathbb{R}^2)$. Thus, a flow-map $\Phi$ that is a solution to the ODE for any $x \in \ohp$: \begin{equation}\label{eq:flow-map}
				\begin{split}
					\frac{\ud}{\ud t} \Phi(t,x) = u(t,\Phi(t,x)), \qquad \Phi(0,x) = x
				\end{split}
			\end{equation} is well-defined on $[0,T]$.	\medskip 
		
		\noindent \textit{Estimates of $\| \tht \|_{C^{1}}$.} Differentiating both sides of \eqref{eq:SQG}, we have \begin{equation}\label{eq:tht-j}
			\begin{split}
				\partial_t \partial_j \tht + u \cdot \nabla \partial_j \tht + \partial_j u \cdot \nabla \tht = 0, \qquad j = 1,2.
			\end{split}
		\end{equation} Applying the flow-map with any given $x \in \ohp \cap B(0;L)$ shows, with $D_{t} = \rd_{t} + u \cdot \nb $, \begin{equation*}
			\begin{split}
				\frac{1}{2} D_t |\partial_j \tht|^2 + (\partial_j u \cdot \nabla \tht )\partial_j \tht = 0,
			\end{split}
		\end{equation*} which implies that \begin{equation*}
			\begin{split}
				D_t |\nabla \tht|^2 \leq 2|\nabla u| |\nabla \tht|^2 \lesssim \| \tht \|_{C^{2,\beta}}^3.
			\end{split}
		\end{equation*}  	\medskip 
		
		\noindent \textit{Estimates of $\| \tht \|_{C^{2}}$.} We estimate the second derivatives of $\tht$ separately according to direction due to lack of regularity of $\partial_2^2u_1$. We can see \begin{equation}\label{eq:tht-11}
			\begin{split}
				D_t(\rd_{1}^{2}\tht) = - 2\rd_1u_1 \rd_{1}^{2}\tht - 2\rd_{1}u_{2} \rd_{21}\tht - \rd_{1}^{2} u_1 \rd_1\tht - \rd_{1}^{2}u_2\rd_2\tht ,
			\end{split}
		\end{equation} 
		\begin{equation}\label{eq:tht-12}
			\begin{split}
				D_t(\rd_{12}\tht) = - 2\rd_2u_1 \rd_{1}^{2}\tht - 2\rd_{2}u_{2} \rd_{2}^{2}\tht - \rd_{12} u_1 \rd_1\tht - \rd_{12}u_2\rd_2\tht ,
			\end{split}
		\end{equation} and
		\begin{equation}\label{eq:tht-22}
			\begin{split}
				D_t(\rd_{2}^{2}\tht) = - 2\rd_2 u_1 \rd_{21}\tht - 2\rd_{2}u_{2} \rd_{2}^{2}\tht - \rd_{2}^{2} u_1 \rd_1\tht - \rd_{2}^{2}u_2\rd_2\tht .
			\end{split}
		\end{equation} From each equation, one can similarly show by the use of \eqref{eq:u-C2} and Remark~\ref{rmk:gm} that \begin{equation*}
			\begin{split}
				D_t |\partial_1^2 \tht|^2 \lesssim \| \tht \|_{C^{2,\beta}}^3, \qquad D_t |\partial_{12} \tht|^2 \lesssim \| \tht \|_{C^{2,\beta}}^3,
			\end{split}
		\end{equation*} and \begin{equation*}
			\begin{split}
				D_t|\partial_2^2\tht|^2 \lesssim \| \tht \|_{C^{2,\beta}}^3 + |\partial_2^2 u_1 \partial_1 \tht||\partial_2^2 \tht|,
			\end{split}
		\end{equation*} respectively. Since the boundary condition implies $\partial_1\tht(x_1,0) = 0$, it follows $|\partial_1\tht(x)| \leq x_2\| \tht \|_{C^2}$ and \begin{equation*}
			\begin{split}
				|\partial_1\tht(x)| \leq 2 \frac{x_2}{1+x_2} \| \tht \|_{C^2}, \qquad x \in \ohp.
			\end{split}
		\end{equation*} Combining with \eqref{eq:u-log}, we have $|\partial_2^2 u_1 \partial_1 \tht| \lesssim \| \tht \|_{C^{2,\beta}}$ and \begin{equation*}
			\begin{split}
				D_t |\partial_1^2 \tht|^2 \lesssim \| \tht \|_{C^{2,\beta}}^3.
			\end{split}
		\end{equation*}  	\medskip 
		
		\noindent \textit{Estimates of $\| \tht \|_{C^{2,\beta}}$.} Take $y \in \ohp \cap B(0;L)$ with $x \neq y$. Note that since we have \begin{equation}\label{eq:flow-lip}
			\begin{split}
				\exp{\left( -\int_0^t \| \nabla u(\tau) \|_{L^{\infty}} \,\ud \tau \right)} {\leq} \frac{|\Phi(t,x) - \Phi(t,y)|}{|x-y|} {\leq} \exp{\left( \int_0^t \| \nabla u(\tau) \|_{L^{\infty}} \,\ud \tau \right)},
			\end{split}
		\end{equation} $|\Phi(t,x) - \Phi(t,y)| \sim |x-y|$ holds on any closed interval. For simplicity, let $\partial_i \partial_j \tht(t,\Phi(t,y)) = \partial_i \partial_j \tht'$ and $\partial_i\partial_ju(t,\Phi(t,y)) = \partial_i\partial_ju'$. Then, it follows from \eqref{eq:tht-11} that \begin{equation*}
			\begin{gathered}
				D_t(\partial_1^2\tht - \partial_1^2\tht') = -2 (\partial_1u_1 - \partial_1u_1')\partial_1^2\tht -2 \partial_1u_1'(\partial_1^2 \tht - \partial_1^2 \tht') - 2(\rd_{1}u_{2} - \rd_{1}u_{2}')\rd_{21}\tht - 2\rd_{1}u_{2}' (\rd_{21}\tht - \rd_{21}\tht') \\
				- (\rd_{1}^{2} u_1 - \rd_1^2 u_1') \rd_1\tht - \rd_{1}^{2} u_1' (\rd_1\tht - \rd_1 \tht') - (\rd_{1}^{2}u_2 - \rd_{1}^{2}u_2') \rd_2\tht - \rd_{1}^{2}u_2' (\rd_2\tht -\rd_2\tht').
			\end{gathered}
		\end{equation*} Multiplying and dividing both sides by $\partial_1^2\tht - \partial_1^2\tht'$ and $|x-y|^{2\beta}$ respectively, we can have by \eqref{eq:u-C2} \begin{equation*}
			\begin{gathered}
				D_t\left( \frac{|\partial_1^2\tht - \partial_1^2\tht'|}{|x-y|^{\bt}} \right)^2 \lesssim \| \tht \|_{C^{2,\beta}}^3,
			\end{gathered}
		\end{equation*} and similarly, it follows from \eqref{eq:tht-12} that \begin{equation*}
			\begin{split}
				D_t\left( \frac{|\partial_{12}\tht - \partial_{12}\tht'|}{|x-y|^{\bt}} \right)^2 \lesssim \| \tht \|_{C^{2,\beta}}^3.
			\end{split}
		\end{equation*} To obtain $\| \partial_2^2 \tht \|_{C^\beta}$ estimate, we see \begin{equation*}
			\begin{gathered}
				D_t(\partial_2^2\tht - \partial_2^2\tht') = -2 (\partial_2u_1 - \partial_2u_1')\partial_{12}\tht -2 \partial_2u_1'(\partial_{12} \tht - \partial_{12} \tht') - 2(\rd_{2}u_{2} - \rd_{2}u_{2}')\rd_{2}^2\tht - 2\rd_{2}u_{2}' (\rd_{2}^2\tht - \rd_{2}^2\tht') \\
				- \rd_{2}^{2} u_1 \rd_1\tht + \rd_{2}^{2} u_1' \rd_1 \tht' - (\rd_{2}^{2}u_2 - \rd_{2}^{2}u_2') \rd_2\tht - \rd_{2}^{2}u_2' (\rd_2\tht -\rd_2\tht').
			\end{gathered}
		\end{equation*} By \eqref{eq:u-C2} and \begin{equation*}
			- \rd_{2}^{2} u_1 \rd_1\tht + \rd_{2}^{2} u_1' \rd_1 \tht' = (-x_2 \rd_2^2 u_1(x) + y_2 \rd_2^2 u_1(y)) \frac{\rd_1\tht(x)}{x_2} - y_2\rd_2^2 u_1(y) (\frac{\rd_1 \tht(x)}{x_2} - \frac{\rd_1 \tht(y)}{y_2}),
		\end{equation*} we arrive at \begin{equation*}
			\begin{gathered}
				D_t\left( \frac{|\partial_{2}^2\tht - \partial_{2}^2\tht'|}{|x-y|^{\bt}} \right)^2 \lesssim \| \tht \|_{C^{2,\beta}}^3 + \frac{|x_2 \rd_2^2 u_1(x) - y_2 \rd_2^2 u_1(y)|}{|x-y|^{\beta}} \left| \frac{\rd_1 \tht(x)}{x_2} \right| \frac{|\rd_{2}^{2} \tht(x) - \rd_2^2 \tht(y)|}{|x-y|^{\beta}} \\
				+ |y_2\rd_2^2 u_1(y)| \frac{|\frac{\rd_1 \tht(x)}{x_2} - \frac{\rd_1 \tht(y)}{y_2}|}{|x-y|^{\beta}} \frac{|\rd_{2}^{2} \tht(x) - \rd_2^2 \tht(y)|}{|x-y|^{\beta}}.
			\end{gathered}
		\end{equation*} To control the last two terms on the right-hand side, we observe \begin{equation*}
		\begin{aligned}
			\frac{\rd_1 \tht(x)}{x_2} = \int_0^1 \rd_{12} \tht(x_1,x_2 \tau) \,\ud \tau,
		\end{aligned}
	\end{equation*} and note \begin{equation*}
		\begin{aligned}
			\left| \frac{\rd_1 \tht(x)}{x_2} \right| \leq \| \tht \|_{C^2}, \qquad \frac{|\frac{\rd_1 \tht(x)}{x_2} - \frac{\rd_1 \tht(y)}{y_2}|}{|x-y|^{\beta}} \leq \int_0^1 \frac{|\rd_{12}\tht(x_1,x_2 \tau) - \rd_{12}\tht(y_1,y_2 \tau)|}{|(x_1,x_2\tau) - (y_1,y_2\tau)|^{\beta}} \,\ud \tau \leq \| \tht \|_{C^{2,\beta}}.
		\end{aligned}
	\end{equation*} With \eqref{eq:u-log1}, we have \begin{equation*}
			\begin{aligned}
				|y_2\rd_2^2 u_1(y)| \frac{|\frac{\rd_1 \tht(x)}{x_2} - \frac{\rd_1 \tht(y)}{y_2}|}{|x-y|^{\beta}} \frac{|\rd_{2}^{2} \tht(x) - \rd_2^2 \tht(y)|}{|x-y|^{\beta}} \leq C(|4\rd_2^2 \tht(y) y_2 \log y_2| + \| \tht \|_{C^{2,\beta}}) \| \tht \|_{C^{2,\beta}}^2 \leq  C\| \tht \|_{C^{2,\beta}}^3.
			\end{aligned}
		\end{equation*} For the remainder term, we assume $x_2 \leq y_2$ without loss of generality and consider \begin{equation*}
		\begin{aligned}
			\frac{|x_2 \rd_2^2 u_1(x) - y_2 \rd_2^2 u_1(y)|}{|x-y|^{\beta}} \leq \frac{|x_2 (\rd_2^2 u_1(x) - \rd_2^2 u_1(y))|}{|x-y|^{\beta}} + \frac{|(x_2-y_2) \rd_2^2 u_1(y)|}{|x-y|^{\beta}}.
		\end{aligned}
	\end{equation*} Note from \eqref{eq:u-log1} that \begin{equation*}
		\begin{aligned}
			\frac{|(x_2-y_2) \rd_2^2 u_1(y)|}{|x-y|^{\beta}} \leq 4(y_2 - x_2)^{1-\beta} |\rd_2^2 \tht(y) \log y_2| + C(y_2 - x_2)^{1-\beta} \| \tht \|_{C^{2,\beta}} \leq C \| \tht \|_{C^{2,\beta}}.
		\end{aligned}
	\end{equation*} Since \eqref{eq:u-log} implies \begin{equation*}
		\begin{gathered}
			|x_2(\rd_2^2 u_1(x) - \rd_2^2 u_1(y))| \leq |x_2(4\rd_2^2 \tht(x) \log x_2 - 4 \rd_2^2 \tht(y) \log y_2)| \\
			+ \left|4x_2(\rd_{2}^2 \tht(x) - \rd_{2}^2 \tht(y)) \log (2|x-y| + \sqrt{4|x-y|^2 + x_2^2})\right| + C \| \tht \|_{C^{2,\beta}}  |x-y|^{\beta} \\
			\leq |(4\rd_2^2 \tht(x) - 4\rd_2^2 \tht(y)) x_2\log x_2| + |4\rd_2 \tht(y) x_2(\log x_2 - \log y_2)| + C|\rd_{2}^2 \tht(x) - \rd_{2}^2 \tht(y)| + C \| \tht \|_{C^{2,\beta}}  |x-y|^{\beta},
		\end{gathered}
	\end{equation*} it follows that \begin{equation*}
		\begin{aligned}
			\frac{|x_2 (\rd_2^2 u_1(x) - \rd_2^2 u_1(y))|}{|x-y|^{\beta}} &\leq \frac{|4\rd_2 \tht(y) x_2(\log x_2 - \log y_2)|}{|x-y|^{\beta}} + C \| \tht \|_{C^{2,\beta}} \\
			&\leq C\| \tht \|_{C^2}\frac{x_2|\log x_2 - \log y_2|}{|x-y|} + C \| \tht \|_{C^{2,\beta}} \\
			&\leq C \| \tht \|_{C^{2,\beta}}.
		\end{aligned}
	\end{equation*} We used the mean value theorem with $x_2 \leq y_2$ in the last inequality. From these estimates, we have \begin{equation*}
		\begin{aligned}
			\| \tht \|_{C^2} \frac{|x_2 \rd_2^2 u_1(x) - y_2 \rd_2^2 u_1(y)|}{|x-y|^{\beta}} \frac{|\rd_{2}^{2} \tht(x) - \rd_2^2 \tht(y)|}{|x-y|^{\beta}}  \leq C\| \tht \|_{C^{2,\beta}}^3,
		\end{aligned}
	\end{equation*} and obtain \begin{equation*}
			\begin{split}
				D_t\left( \frac{|\partial_{2}^2\tht - \partial_{2}^2\tht'|}{|x-y|^{\bt}} \right)^2 \lesssim \| \tht \|_{C^{2,\beta}}^3.
			\end{split}
		\end{equation*}

		Collecting all the estimates and taking the supremum in $x, y$, we obtain \begin{equation}\label{eq:C2beta-apriori}
				\begin{split}
					\frac{d}{dt} \| \tht \|_{C^{2,\beta}}^2 \lesssim \| \tht \|_{C^{2,\beta}}^3.
				\end{split}
			\end{equation} We may take $T>0$ smaller in a way that \begin{equation}\label{eq:C2beta-apriori2}
				\begin{split}
					\sup_{ t \in [0,T] }\| \tht(t) \|_{C^{2,\beta}} \le 2\| \tht_{0} \|_{C^{2,\beta}}.
				\end{split}
			\end{equation}
			
			\subsubsection{Existence}
			
			The existence of a $L^{\infty}([0,T];C^{2,\bt})$ solution can be proved by the following iteration scheme: we set $\tht^{(0)} \equiv \tht_{0}$ on $[0,T]$ and given $n\ge 0$ and $\tht^{(n)} \in L^{\infty}([0,T];C^{2,\bt})$ supported in $B(0;L)$, we define $\tht^{(n+1)}$ as the solution of the transport equation \begin{equation}\label{eq:iteration}
				\begin{split}
					\rd_{t} \tht^{(n+1)} + u^{(n)}\cdot\nb \tht^{(n+1)} = 0 
				\end{split}
			\end{equation} on $t \in [0,T]$ with initial data $\tht_{0}$. Here, $u^{(n)}$ is the velocity corresponding to $\tht^{(n)}$, and therefore satisfies the estimates \eqref{eq:u-C2} and \eqref{eq:u-log} from Lemma \ref{lem:vel}. In particular, for $x_{2}>0$, the velocity is $C^{2}$ and the equations for the second derivatives for $\tht^{(n+1)}$ are valid in $\bbR^{2}_{+}$: \eqref{eq:tht-11}, \eqref{eq:tht-12}, \eqref{eq:tht-22} hold for $\tht$ and $u$ replaced by $\tht^{(n+1)}$ and $u^{(n)}$, respectively. Rewriting $\rd_{22}u_{1}^{(n)} \rd_{1}\tht^{(n+1)} $ by $x_{2}\rd_{22}u_{1}^{(n)} x_{2}^{-1}\rd_{1}\tht^{(n+1)} $ and proceeding as in the proof of $C^{2}$ a priori estimate, we obtain that $\tht^{(n+1)}$ is $C^{2}$ up to the boundary: $\tht^{(n+1)} \in L^\infty([0,T];C^{2}(\ohp))$. Therefore, $\nb^2 \tht^{(n+1)}$ is defined pointwise in $\ohp$ and we can follow the above proof of a priori estimate in $C^{2,\bt}$ and derive \begin{equation*}
				\begin{split}
					\frac{d}{dt} \| \tht^{(n+1)} \|_{C^{2,\beta}}^2 \lesssim \| \tht^{(n)} \|_{C^{2,\beta}} \| \tht^{(n+1)} \|_{C^{2,\beta}}^2.
				\end{split}
			\end{equation*} with the same implicit constant as in \eqref{eq:C2beta-apriori} (in particular, independent of $n$). Then, recalling \eqref{eq:C2beta-apriori2}, we obtain \begin{equation*}
				\begin{split}
					\sup_{n\ge0}	\sup_{ t \in [0,T] }\| \tht^{(n)}(t) \|_{C^{2,\beta}} \le 2\| \tht_{0} \|_{C^{2,\beta}}.
				\end{split}
			\end{equation*} Using this uniform bound, taking $T>0$ smaller if necessary, one can show that  $\left\{ \tht^{(n)} \right\}_{n\ge0}$ is a Cauchy sequence in $C([0,T];L^2(\ohp))$. Using the uniform $C^{2,\bt}$ bound again, we have that $\left\{ \tht^{(n)}, u^{(n)} \right\}_{n\ge0}$ is Cauchy in $C^{1}([0,T]\times \ohp)$. Taking the limit $n\to\infty$ in \eqref{eq:iteration}, we see that $\tht^{(n)}\to \tht$ in $C^{1}([0,T]\times \ohp)$, where $\tht$ is the unique $C^{1,\bt}$ solution corresponding to the initial data $\tht_{0}$ (remarked in Subsection \ref{subsec:wellposed-rmk}). The convergence of $\tht^{(n)}\to \tht$  can be upgraded to $L^{\infty}([0,T];C^{2,\gmm})$ with any $0<\gmm<\bt$. Then, from the pointwise convergence of second derivatives in space, we see that $\tht$ inherits the uniform $C^{2,\bt}$ bound. This gives that $\tht \in L^{\infty}([0,T];C^{2,\bt})$ with $\sup_{ t \in [0,T] }\| \tht(t) \|_{C^{2,\beta}} \le 2\| \tht_{0} \|_{C^{2,\beta}}.$

		\subsection{Propagation of $W^{3,p}_0$  }\label{subsec:W3p}
		In this section, we prove local wellposedness of \eqref{eq:SQG} in $W^{3,p}_{0}$ for any $1 < p < \infty$. In this section, we let $C>0$ depend on $p$ as well.
		
		\subsubsection{A priori estimate}\label{subsec:a}
		We fix $1 < p < \infty$ and let $\tht$ be a $W^{3,p}_{0}$-solution to \eqref{eq:SQG} for some $T>0$. Then, we can use the flow-map $\Phi(t,x)$ satisfying \eqref{eq:flow-map} for all $x \in \ohp$ and Lemma~\ref{lem:vel-W3p} on the interval $[0,T]$. For convenience, we recall the estimates from Lemma~\ref{lem:vel-W3p} here: \begin{equation}\label{eq:vel-reg2}
			\begin{aligned}
				\| \nabla u \|_{L^{\infty}(\bbR^2)} + \| \nabla^2 u \|_{L^{2p}(\bbR^2)} + \| \nabla^2 \partial_1 u \|_{L^{p}(\bbR^{2})} + 	\| \rd_{222}u_1 - 2H[\tht] \|_{L^{p}(\bbR^{2})} \leq C \| \tht \|_{W^{3,p}(\bbR^2_{+})},
			\end{aligned}
		\end{equation}  where $H[\tht]$ satisfies \begin{equation}\label{eq:H2p2}
			\begin{aligned}
				\| x_2 H[\tht] \|_{L^{2p}(\bbR^2)}  \leq C \|   \tht \|_{W^{3,p}(\bbR^2_{+})}.
			\end{aligned}
		\end{equation} We begin by recalling that \eqref{eq:SQG} conserve the $L^p$-norm  for all $1 \leq p \leq \infty$: $\| \tht(t) \|_{L^p} = \| \tht_0 \|_{L^p}.$	\medskip 
		
		\noindent \textit{Estimates of $\| \tht \|_{W^{2,p}}$.} For $1 \leq i,\,j \leq 2$, we have \begin{equation*}
			\frac{1}{p} D_t|\rd_{ij}\tht|^p + (\rd_{i} u \cdot \nabla \rd_{j} \tht + \rd_{j} u \cdot \nabla \rd_{i} \tht + \rd_{ij} u \cdot \nabla \tht) |\rd_{ij} \tht|^{p-2} \rd_{ij}\tht = 0.
		\end{equation*} Using \eqref{eq:vel-reg} with the continuous embedding $W^{1,p}(\bbR^2_{+}) \hookrightarrow L^{2p}(\bbR^2_{+})$, we can see \begin{equation*}
			\frac{d}{dt} \| \nabla^2 \tht \|_{L^p}^p \leq C \left( \| \nabla u \|_{L^{\infty}} \| \nabla^2 \tht \|_{L^p}^p + \| \nabla^2 u \|_{L^{2p}} \| \nabla \tht \|_{L^{2p}} \| \nabla^2 \tht \|_{L^p}^{p-1} \right) \le C \| \tht \|_{W^{3,p}}^{p+1}.
		\end{equation*}  	\medskip 
		
		\noindent \textit{Estimates of $\| \tht \|_{W^{3,p}}$.} Among the third order derivatives of $\tht$, we consider the most difficult term, namely $\rd_{2}^{3}\tht$: 
		\begin{equation*}
			\frac{1}{p} D_t |\rd_{2}^3 \tht|^p + (\rd_{2}^3u \cdot \nabla \tht + 3\rd_{2}^2u \cdot \nabla \rd_{2} \tht + 3\rd_{2} u \cdot \nabla \rd_{2}^2\tht)|\rd_2^3 \tht|^{p-2} \rd_{2}^3 \tht = 0.
		\end{equation*} Integrating in space, we obtain that \begin{equation*}
			\begin{split}
				\frac{d}{dt} \nrm{\rd_{2}^{3} \tht}_{L^{p}}^{p} \le C  \nrm{\rd_{2}^{3} \tht}_{L^{p}}^{p-1} \left(  \nrm{ \rd_{2}^3u \cdot \nabla \tht }_{L^p} + \nrm{\rd_{2}^2u \cdot \nabla \rd_{2} \tht}_{L^p} + \nrm{\rd_{2} u \cdot \nabla \rd_{2}^2\tht }_{L^p} \right). 
			\end{split}
		\end{equation*} The last two terms are straightforward to bound, just using \eqref{eq:vel-reg2}: \begin{equation*}
			\begin{split}
				\nrm{\rd_{2} u \cdot \nabla \rd_{2}^2\tht }_{L^p} \le \nrm{\rd_{2} u }_{L^\infty} \nrm{\nabla \rd_{2}^2\tht }_{L^p} \le C\nrm{\tht}_{W^{3,p}}^{2}, 
			\end{split}
		\end{equation*}\begin{equation*}
			\begin{split}
				\nrm{\rd_{2}^2u \cdot \nabla \rd_{2} \tht}_{L^p} \le C\nrm{\rd_{2}^2u }_{L^{2p}} \nrm{\nabla \rd_{2} \tht}_{L^{2p}} \le C\nrm{\tht}_{W^{3,p}}^{2}. 
			\end{split}
		\end{equation*} Then, to treat the remaining term, we rewrite it as \begin{equation*}
			\begin{split}
				\rd_{2}^3u \cdot \nabla \tht = 2x_{2}H[\tht] \frac{\rd_{1}\tht}{x_{2}} +  (\rd_{2}^3u_1 - 2H[\tht]) \rd_{1}\tht -  \rd_{2}^{2}\rd_{1}u_{1} \rd_{2}\tht. 
			\end{split}
		\end{equation*} Again, the last two terms can be bounded simply using \eqref{eq:vel-reg2}: \begin{equation*}
			\begin{split}
				\nrm{ (\rd_{2}^3u_1 - 2H[\tht]) \rd_{1}\tht -  \rd_{2}^{2}\rd_{1}u_{1} \rd_{2}\tht}_{L^{p}} \le  C\nrm{\tht}_{W^{3,p}}^{2}. 
			\end{split}
		\end{equation*} Finally, we can bound using \eqref{eq:H2p2} and Hardy's inequality that \begin{equation*}
			\begin{split}
				\left\Vert x_{2}H[\tht] \frac{\rd_{1}\tht}{x_{2}} \right\Vert_{L^p} \le C\left\Vert x_{2}H[\tht] \right\Vert_{L^{2p}}\left\Vert \frac{\rd_{1}\tht}{x_{2}} \right\Vert_{L^{2p}} \le C\left\Vert x_{2}H[\tht] \right\Vert_{L^{2p}}\left\Vert  {\rd_{12}\tht}  \right\Vert_{L^{2p}} \le C\nrm{\tht}_{W^{3,p}}^{2}.
			\end{split}
		\end{equation*}This shows that \begin{equation*}
			\begin{split}
				\frac{d}{dt} \nrm{\rd_{2}^{3} \tht}_{L^{p}}^{p} \le C   \nrm{ \tht}_{W^{3,p}}^{p+1}.
			\end{split}
		\end{equation*} Estimating other third order derivatives of $\tht$ (which then involves at least one factor of $\rd_{1}$) is strictly easier, and we omit the proof. Then, collecting the estimates, we can conclude  \begin{equation*}
			\begin{split}
				\frac{d}{dt} \nrm{ \tht}_{W^{3,p}}^{p} \le C   \nrm{ \tht}_{W^{3,p}}^{p+1}.
			\end{split}
		\end{equation*} 

        \subsubsection{Existence}

        Given the a priori estimates, we sketch the existence of a $W^{3,p}$ solution. While we expect that there are several ways of doing this (for instance, one may adapt the scheme introduced in \cite{CNgu2}), we present a method which uses a mollified velocity. Namely, given an initial datum $\tht_{0} \in W^{3,p}(\ohp)$ and some $\varepsilon>0$, we consider the transport equation \begin{equation}\label{eq:mollify}
        	\begin{split}
        		\rd_t \tht + u^{\varepsilon}\cdot\nb \tht = 0, \qquad \tht(t=0) = \tht_{0}. 
        	\end{split}
        \end{equation} where we define \begin{equation}\label{eq:u-eps}
        	\begin{split}
        		u^{\varepsilon}(t,x) & = \nb^\perp_{x}\left[ \int_{\bbR^2} \frac1{ (|x-y|^2 + \varep^2)^{1/2} } \overline{\tht}(t,y) dy \right] \\
        		&=  \int_{ \bbR^2_+ } \left[ \frac{(x-y)^\perp}{(|x-y|^2+\varep^2)^{3/2}} - \frac{(x-\bar{y})^\perp}{(|x-\bar{y}|^2+\varep^2)^{3/2}}   \right] \tht(t,y) \, \ud y. 
        	\end{split}
        \end{equation}
        For compactly supported $\tht \in L^{\infty}$, $u^\varep$ is a smooth divergence-free vector field for each $\varep>0$. 
        
        Using that $u^\varep$ is smooth (whenever $\tht$ is bounded and compactly supported), it can be shown that \eqref{eq:mollify} has a unique solution for all times, which belongs to $L^{\infty}([0,T];W^{3,p})$ for any $T>0$. Let us denote this unique solution by $\tht^{(\varep)}(t,x)$. (One may also mollify the initial data as well, which results in a $C^{\infty}$ solution for each $\varep>0$.) 
        
        Since $\tht^{(\varep)} \in   L^{\infty}([0,T];W^{3,p})$, we can justify the $W^{3,p}$ a priori estimate given in the previous section. Here, the point is that, by repeating the proof of Lemma \ref{lem:vel-W3p} for $u^\varep$ instead of $u$, we can get all the same inequalities from Lemma \ref{lem:vel-W3p} for $u^\varep$ with $C$ independent of $\varep$, where $H[\tht]$ is now replaced by 
        \begin{equation*}
        	\begin{split}
        		H^{\varep}[\tht^{(\varep)}](x) := \int_{\bbR} \frac{x_2}{(|x-(y_1,0)|^2 + \varep^2 )^{3/2}} \partial_{22} \tht^{(\varep)} (y_1,0) \,\ud y_1.
        	\end{split}
        \end{equation*}
        That is, the estimates \begin{equation*} 
        	\begin{split}
        		\| \rd_{222}u_1^{\varep} - 2H^{\varep}[\tht^{(\varep)} ] \|_{L^{p}(\bbR^{2})} \leq C\| \rd_{222} \tht^{(\varep)}  \|_{L^p(\bbR^2_{+})}, \qquad 	\| x_2 H^{\varep}[\tht^{(\varep)} ] \|_{L^{2p}(\bbR^2)}  \leq C \| \partial_{22} \tht^{(\varep)}  \|_{W^{1,p}(\bbR^2_{+})}
        	\end{split}
        \end{equation*}   hold with some $C>0$ independent of $\varep>0$. In particular, we can find $T>0$, depending only on $\nrm{\tht_{0}}_{W^{3,p}}$, such that the sequence of solutions $\{ \tht^{(\varep)} \}_{\varep>0}$ is bounded uniformly in the space $L^{\infty}([0,T];W^{3,p})$. This implies that the sequence of time derivatives $\{ \rd_t\tht^{(\varep)} \}_{\varep>0}$ is uniformly bounded in $L^{\infty}([0,T];W^{2,p})$. Therefore, by appealing to the Aubin--Lions lemma, we can extract a subsequence which is strongly (in $C_{t}W^{2,p}$) convergent to some $\tht$ belonging to $L^{\infty}([0,T];W^{3,p})$ and satisfying $\tht(t=0)=\tht_{0}$. Strong convergence in $W^{2,p}$ shows that $\tht$ is a solution to \eqref{eq:SQG}. 
        
        {It only remains to prove that the solution $\tht$ actually belongs to $C([0,T];W^{3,p})$, and for this we may apply the argument in Bedrossian--Vicol \cite[pp. 44-45]{BV}. First, the same argument shows existence of the solution $\tht \in L^{\infty}([-T,T];W^{3,p})$, and by time translation symmetry of \eqref{eq:SQG}, it suffices to prove continuity of the solution at $t = 0 $. Furthermore, from time reversal symmetry, we can just check the right continuity at $t = 0$. On $(-T,T)$, using that $\rd_t\tht = -u\cdot\nb\tht$, we have that $\rd_t\tht \in L^{\infty}(-T,T;C^{\gmm})$ for some $\gmm>0$. This implies that $\tht$ is weakly continuous in time with values in $W^{3,p}$; namely, for any test function $\varphi \in (W^{3,p})^*$, the duality pairing $\brk{ \tht, \varphi }$ is a continuous function of $t \in (-T,T)$. This implies in particular that the norm may only drop at $t = 0$; \begin{equation*}
        		\begin{split}
        			\liminf_{t\to 0^+} \nrm{\tht(t,\cdot)}_{W^{3,p}} \ge \nrm{\tht_{0}}_{W^{3,p}}. 
        		\end{split}
        \end{equation*} Finally, the opposite inequality \begin{equation*}
        \begin{split}
        	\limsup_{t\to 0^+} \nrm{\tht(t,\cdot)}_{W^{3,p}} \le \nrm{\tht_{0}}_{W^{3,p}}
        \end{split}
    \end{equation*}  follows from the a priori estimate for the approximating sequence $\tht^{(\varep)}$: uniformly in $\varep$, we have \begin{equation*}
    \begin{split}
    	\nrm{\tht(t,\cdot)}_{W^{3,p}} \le \liminf_{\varepsilon\to 0^+}\nrm{\tht^{(\varep)}(t,\cdot)}_{W^{3,p}} \le (1+O(t))\nrm{\tht_{0}}_{W^{3,p}}. 
    \end{split}
\end{equation*}This finishes the proof of existence. }

		\section{High regularity illposedness}\label{sec:illposed}

		\begin{proof}[Proof of Theorem \ref{thm:nonexist}]
			
			We take initial data $\tht_{0} \in C^{\infty}_{0}(\overline{\bbR^2_{+}})$ supported in {$B(0;1)$} satisfying  \begin{equation*}
				\begin{split}
					\tht_0(x) = A x_1x_2 + B x_2^2, \quad x \in {B(0;r_{0})}
				\end{split}
			\end{equation*} for some $1/2 > r_{0}>0$. {(As we shall see, the proof is easily adapted to the more general case when we have $\rd_{12}\tht_{0}(0)$ and   $\rd_{22}\tht_{0}(0)$ are both nonzero.)} There exist $T>0$ and a unique solution $\tht \in L^\infty([0,T);C^{2,\bt}_{0}(\overline{\bbR^2_{+}}))$ (for any $\bt<1$) with $\tht(t=0) = \tht_{0}$. Towards a contradiction, we assume that there is some $0<\dlt \le T$ such that $\tht \in L^\infty([0,\dlt];C^3_0(\ohp))$, and set \begin{equation}\label{eq:def-M}
				\begin{split}
					M = \sup_{ t \in [0,\dlt] } \nrm{\tht(t,\cdot)}_{W^{3,\infty}}. 
				\end{split}
			\end{equation} On the other hand, we shall simply bound by $C$ the quantities that depend only on the $C^{2,\bt}$-norm of $\tht$. 
			
			Now, we take $0 < |x_{1}| \le r_{0}/10$ and set $x = (x_1,x_2)$ and $y = (x_1,0)$ where $|x_{1}| \le |x_{2}|$. We shall consider the derivatives of $\tht$ along the trajectories of $x$ and $y$, which are denoted by $\Phi(t,x)$ and $\Phi(t,y)$, respectively.

			\medskip 
			
			\noindent \textit{Estimates of $\Phi$ and $\nb\tht$.} Before we proceed, let us first obtain a few simple estimates on $\Phi = (\Phi_1,\Phi_2)$. To begin with, we note that $$\partial_1u_2 (x_1,0) = 0 = \partial_2 u_1(x_1,0) .$$ They follow from the fact that (viewed as functions on $\bbR^{2}$) $u_2$ and $u_1$ are respectively odd and even in $x_{2}$. Then, using that $u_{2}$ vanishes on the boundary, \begin{equation*}
				\begin{split}
					\left| \frac{\ud}{\ud t} \Phi_{2}(t,x) \right| = \left| u_{2}(t,\Phi(t,x)) \right| \le \nrm{\nb u_{2}}_{L^\infty} \left| \Phi_{2}(t,x) \right|,
				\end{split}
			\end{equation*} which gives \begin{equation}\label{eq:Phi-2-ratio}
				\begin{split}
					\frac12 \le \frac{\Phi_{2}(t,x)}{x_{2}} \le \frac32 , \qquad t \in [0,\dlt]
				\end{split}
			\end{equation} by taking $\dlt > 0 $ smaller if necessary. Next, using that $\Phi_{2}(t,y) = 0$ and $\Phi_{1}(0,x) = x_{1} = \Phi_{1}(0,y)$, \begin{equation*}
				\begin{split}
					\left| \frac{\ud}{\ud t} \left(\Phi_{1}(t,x) - \Phi_{1}(t,y) \right) \right| \le C\nrm{\nb u_{1}}_{L^\infty} \left(  \Phi_{2}(t,x) + \left|\Phi_{1}(t,x) - \Phi_{1}(t,y)\right| \right), 
				\end{split}
			\end{equation*} gives after using Gronwall's inequality and again taking $\dlt > 0 $ smaller if necessary that \begin{equation}\label{eq:Phi-1-diff}
				\begin{split}
					\left|\Phi_{1}(t,x) - \Phi_{1}(t,y)\right| \le Ctx_{2} \le x_{2}, \qquad t \in [0,\dlt]. 
				\end{split}
			\end{equation} Therefore, using \eqref{eq:Phi-1-diff} and \eqref{eq:Phi-2-ratio} we get \begin{equation}\label{eq:Phi-diff}
				\begin{split}
					\left|\Phi(t,x) - \Phi(t,y)\right| \le 2x_{2} , \qquad t \in [0,\dlt]. 
				\end{split}
			\end{equation} 
			Next, using that $\rd_1u_2$ vanishes on the boundary, we write 
			\begin{equation}
				\begin{aligned}
					\frac{\ud}{\ud t} \partial_1 \theta (t,\Phi(t,x)) = (-\partial_1 u_1 \partial_1 \theta - \frac{\partial_1u_2}{x_2} \partial_2 \theta x_2)(t,\Phi(t,x)).
				\end{aligned}
			\end{equation} Using boundedness of $\nrm{\rd_{12}u_{2}}_{L^\infty}$ and $\rd_{1}\tht(t=0) = A x_{2}$, we obtain \begin{equation}\label{eq:rd-1-tht}
				\begin{split}
					\left|  \partial_1 \theta (t,\Phi(t,x)) - Ax_{2} \right| \le \frac{A}{2} x_{2} , \qquad t \in [0,\dlt] 
				\end{split}
			\end{equation} again taking $\dlt>0$ smaller. Next, we have 
			\begin{equation}
				\begin{aligned}
					\frac{\ud}{\ud t} \partial_2 \theta (t,\Phi(t,x))  =( - {\partial_2u_1} \partial_1 \theta - \partial_2 u_2 \partial_2 \theta)(t,\Phi(t,x))
				\end{aligned}
			\end{equation} and using \eqref{eq:rd-1-tht} with boundedness of $\nrm{\rd_2u_2}_{L^\infty}$ gives \begin{equation}\label{eq:rd-2-tht}
				\begin{split}
					| \partial_2 \theta (t,\Phi(t,x)) | \le C\left( | \partial_2 \theta (0,x) | + x_{2} \right) \le C(|x_{1}| + x_{2}), \qquad t \in [0,\dlt]. 
				\end{split}
			\end{equation}

			\medskip 
			
			\noindent \textit{Estimates of $\nb^2\tht$.} With these preparations, we   consider  $\rd_{2}^{2}\tht$ along the flow. From \eqref{eq:tht-22}, we have that \begin{equation*}
				\begin{split}
					\frac{\ud}{\ud t} \partial_2^2 \theta(t,\Phi(t,x)) = \left( - 2\rd_2 u_1 \rd_{21}\tht - 2\rd_{2}u_{2} \rd_{2}^{2}\tht - {\rd_{2}^{2} u_1 }\rd_1\tht - \rd_{2}^{2}u_2\rd_2\tht \right) (t,\Phi(t,x)) .
				\end{split}
			\end{equation*} All the second derivatives of $u$ are bounded, except for $\rd_2^2 u_1$. For this term, we estimate \begin{equation*}
				\begin{split}
					\left| {\rd_{2}^{2} u_1 }\rd_1\tht  \right| \le C \nrm{ x_{2}  \rd_{2}^{2} u_1 }_{L^\infty} \nrm{x_{2}^{-1}\rd_1\tht }_{L^\infty} \le C
				\end{split}
			\end{equation*} using \eqref{eq:u-log} and \eqref{eq:rd-1-tht}. Then we can obtain \begin{equation}\label{eq:rd-2-2-tht}
				\begin{split}
					\left| 	\frac{\ud}{\ud t} \partial_2^2 \theta(t,\Phi(t,x)) \right| \le C \quad \implies \quad \rd_2^2\tht(t,\Phi(t,x)) \ge \rd_2^2\tht(0,x) - Ct = 2B - Ct \ge B
				\end{split}
			\end{equation} by shrinking $\dlt>0$ again if necessary. Proceeding similarly and using $\rd_{12}\tht(0,x) = A$, we can prove that \begin{equation}\label{eq:rd-1-2-tht}
				\begin{split}
					\rd_{12}\tht(t,\Phi(t,x)) \ge \frac{A}{2}, \qquad t \in [0,\dlt]. 
				\end{split}
			\end{equation} Since the proof is only simpler, we omit the argument. 
			We are ready to complete the proof, by studying the evolution of difference of $\rd_{2}^{2}\tht$ along two particle trajectories. We write 
			\begin{equation*}
				\begin{split}
					\frac{\ud}{\ud t} \left( \partial_2^2 \theta(t,\Phi(t,x)) - \partial_2^2 \theta(t,\Phi(t,y)) \right) & = (- 2\rd_2 u_1 \rd_{21}\tht )(t,\Phi(t,x)) - (- 2\rd_2 u_1 \rd_{21}\tht )(t,\Phi(t,y)) \\
					& \qquad ( - 2\rd_{2}u_{2} \rd_{2}^{2}\tht)(t,\Phi(t,x)) - ( - 2\rd_{2}u_{2} \rd_{2}^{2}\tht)(t,\Phi(t,y)) \\
					&\qquad ( - \rd_{2}^{2} u_1 \rd_1\tht )(t,\Phi(t,x)) - ( - \rd_{2}^{2} u_1 \rd_1\tht )(t,\Phi(t,y)) \\
					&\qquad ( - \rd_{2}^{2}u_2\rd_2\tht )(t,\Phi(t,x)) - ( - \rd_{2}^{2}u_2\rd_2\tht )(t,\Phi(t,y)) \\
					& =: I + II + III + IV . 
				\end{split}
			\end{equation*}
			We now estimate the differences term by term. 
			
			\medskip 
			
			\noindent \textit{Estimate of $II$.} Simply using the mean value theorem, we have that \begin{equation*}
				\begin{split}
					\left| ( - 2\rd_{2}u_{2} \rd_{2}^{2}\tht)(t,\Phi(t,x)) - ( - 2\rd_{2}u_{2} \rd_{2}^{2}\tht)(t,\Phi(t,y)) \right| &\le C\nrm{ \nb ( \rd_2u_2 \rd_{2}^{2}\tht ) }_{L^\infty} |\Phi(t,x) - \Phi(t,y)| \\
					& \le C\left( \nrm{ \rd_1 u_1 }_{W^{1,\infty}} \nrm{\tht}_{W^{3,\infty}} \right) |\Phi(t,x) - \Phi(t,y)| \\
					& \le CM |\Phi(t,x) - \Phi(t,y)|, 
				\end{split}
			\end{equation*} using $\rd_2u_2 = -\rd_1u_1$.

			\medskip 
			
			\noindent \textit{Estimate of $IV$.} We first rewrite $IV$ as \begin{equation*}
				\begin{split}
					( - \rd_{2}^{2}u_2\rd_2\tht )(t,\Phi(t,x)) - ( - \rd_{2}^{2}u_2\rd_2\tht )(t,\Phi(t,y)) & = - \left( \partial_2^2 u_2(t,\Phi(t,x)) - \partial_2^2 u_2(t,\Phi(t,y)) \right) \partial_2 \theta(t,\Phi(t,x)) \\
					& \qquad + \rd_{2}^{2} u_{2} (t,\Phi(t,y)) \left( \rd_2\tht(t,\Phi(t,y)) - \rd_2\tht(t,\Phi(t,x)) \right) \\ 
					& =: IV_1 + IV_2. 
				\end{split}
			\end{equation*} Using $\rd_{2}^{2}u_2 = -\rd_{2} \rd_{1} u_{1}$, the term $IV_2$ is simply bounded by \begin{equation*}
				\begin{split}
					\left| IV_2 \right| \le C\nrm{ \rd_{1} u_{1}}_{W^{1,\infty}} \nrm{\tht}_{W^{2,\infty}} |\Phi(t,x) - \Phi(t,y)| \le C|\Phi(t,x) - \Phi(t,y)|. 
				\end{split}
			\end{equation*}
			To bound the other term, we use the log-Lipschitz estimate \eqref{eq:log-Lipschitz} together with \eqref{eq:rd-2-tht}:
			\begin{equation}
				\begin{aligned}
					\left| \partial_2^2u_2(t,x) - \partial_2^2u_2(t,y) \right| \leq CM |x-y| \log \left( 10 +\frac{1}{|x-y|} \right)
				\end{aligned}
			\end{equation} to obtain that \begin{equation*}
				\begin{split}
					\left|IV_1\right| &\le CM(|x_1| + x_2) |\Phi(t,x) - \Phi(t,y)|\log\left( 10 + \frac{1}{|\Phi(t,x) - \Phi(t,y)|}\right)  \\ 
					& \le CMx_{2}|\Phi(t,x) - \Phi(t,y)|\log\left( 10 + \frac{C}{x_{2}}\right) \le CM|\Phi(t,x) - \Phi(t,y)|
				\end{split}
			\end{equation*} using $|x_1| \le  x_2$  and \eqref{eq:Phi-diff}.

			\medskip 
			
			\noindent \textit{Estimate of $I$ and $III$.} We observe that $\rd_{2}^{2} u_{1} \rd_{1}\tht$ and $\rd_{2} u_{1}$ vanish on the boundary, so that $III = -\rd_{2}^{2} u_{1} \rd_{1}\tht (t,\Phi(t,x))$ and similarly, $I = - 2\rd_2 u_1 \rd_{21}\tht (t,\Phi(t,x))$. We now use \eqref{eq:u-log1} and \eqref{eq:rd-1-tht} to replace $\rd_2^2 u_1$ by its main term (with $\Phi = \Phi(t,x)$) \begin{equation*}
				\begin{split}
					III \ge C_* \rd_{22}\tht(t,\Phi) \log \frac{1}{\Phi_{2}} \rd_{1}\tht(t,\Phi) - Cx_{2} \ge \frac{C_*}{2} A B x_{2} \log \frac{1}{x_2} - C x_{2}  
				\end{split}
			\end{equation*} where we have used lower bounds from \eqref{eq:rd-2-2-tht} and \eqref{eq:rd-1-tht} as well as $\log \frac{1}{\Phi_{2}}  = \log \frac{1}{x_{2}} + \log \frac{x_2}{\Phi_{2}} $, with the latter being uniformly bounded by an absolute constant. Regarding $I$, we can first integrate \eqref{eq:u-log} in $x_{2}$ to obtain a lower bound \begin{equation*}
				\begin{split}
					-\rd_{2} u_{1}(t,\Phi) \ge C_* \Phi_2  \log \frac{1}{\Phi_{2}} - C\Phi_2. 
				\end{split}
			\end{equation*} Proceeding similarly and using \eqref{eq:rd-1-2-tht} this time, we obtain a lower bound \begin{equation*}
				\begin{split}
					I \ge {C_*} A B x_{2} \log \frac{1}{x_2} - C x_{2}.
				\end{split}
			\end{equation*}

			\medskip 
			
			\noindent \textit{Collecting the estimates.} From the above, we deduce that \begin{equation*}
				\begin{split}
					\frac{\ud}{\ud t} \left( \partial_2^2 \theta(t,\Phi(t,x)) - \partial_2^2 \theta(t,\Phi(t,y)) \right)  & \ge \frac{3C_*}{2} A B x_{2} \log \frac{1}{x_2} - C x_{2}  - C(1+M)|\Phi(t,x) - \Phi(t,y)| \\
					&  \ge \frac{3C_*}{2} A B x_{2} \log \frac{1}{x_2} - C (1 +M) x_{2} 
				\end{split}
			\end{equation*} using \eqref{eq:Phi-diff} in the last inequality. Observe that when $t = 0$, $\partial_2^2 \theta(t,\Phi(t,x)) - \partial_2^2 \theta(t,\Phi(t,y))$ vanishes. Therefore, \begin{equation*}
				\begin{split}
					\partial_2^2 \theta(t,\Phi(t,x)) - \partial_2^2 \theta(t,\Phi(t,y)) \ge (\frac{3C_*}{2} A B x_{2} \log \frac{1}{x_2} - C (1 +M) x_{2} )t, \qquad t \in [0,\dlt]. 
				\end{split}
			\end{equation*} However, from the contradiction hypothesis we have an upper bound \begin{equation*}
				\begin{split}
					\left| 	\partial_2^2 \theta(t,\Phi(t,x)) - \partial_2^2 \theta(t,\Phi(t,y))  \right| \le M|\Phi(t,x) - \Phi(t,y)| \le CM x_{2}, \qquad t \in [0,\dlt]. 
				\end{split}
			\end{equation*} Dividing by $x_{2}$, we obtain \begin{equation*}
				\begin{split}
					CM \ge (\frac{3C_*}{2} A B  \log \frac{1}{x_2} - C (1 +M) )t
				\end{split}
			\end{equation*} which gives a contradiction by fixing some $t^* \in (0,\dlt]$ and taking $x_{2} \to 0^+$. This finishes the proof. 
		\end{proof} 
		

	\subsection*{Acknowledgments} IJ has been supported by the NRF grant from the Korea government (MSIT), No. 2022R1C1C1011051, RS-2024-00406821. JK was supported by the National Research Foundation of Korea(NRF) grant funded by the Korea government(MSIT) (No. RS-2024-00360798).

		\providecommand{\bysame}{\leavevmode\hbox to3em{\hrulefill}\thinspace}
		\providecommand{\MR}{\relax\ifhmode\unskip\space\fi MR }
		\providecommand{\MRhref}[2]{%
			\href{http://www.ams.org/mathscinet-getitem?mr=#1}{#2}
		}
		\providecommand{\href}[2]{#2}

	\end{document}